\documentclass[12pt, reqno]{amsart}
\usepackage{amsmath,amssymb,amsthm}
\allowdisplaybreaks
\usepackage{amsmath, amssymb}
\usepackage{amsthm, amsfonts, mathrsfs}
\usepackage{mathptmx}
\usepackage{fullpage}
\usepackage{amsfonts,graphicx}
\usepackage{dutchcal}
\numberwithin{equation}{section}
\usepackage[colorlinks=true, pdfstartview=FitV, linkcolor=blue, citecolor=blue, urlcolor=blue]{hyperref}

\newtheorem {theorem}{Theorem}[section]
\newtheorem {lemma}[theorem]{{\bf Lemma}}

\theoremstyle{remark}
\newtheorem {remark}{{\bf Remark}}[section]

\newtheorem {definition}{{\bf Definition}}[section]
\theoremstyle{plain} \numberwithin {equation}{section}

\newcommand{\p}{{\mathbb P}}
\newcommand{\q}{{\mathbb Q}}

\newcommand{\R}{{\mathbb R}}
\def\div{ \hbox{\rm div}\,  }
\newcommand\Z{{\mathbb{Z}}}
\def\la{\Lambda}
\def\nn{\nonumber}
\def\ta{\theta}

\def\ddk{\dot \Delta_k}

\def\G{ \mathbf{G} }
\def\u{ \mathbf{u} }

\def\T{ \mathbb{T} }
\def\ea{ \mathcal E_\infty(t) }
\def\eb{  \mathcal E_1^\lambda(t) }

\begin{document}
\title{ { Large global solutions to the Oldroyd-B model with dissipation}}

\author{Tao Liang}

\address{ School of Mathematics,
	South China University of Technology,
	Guangzhou, 510640, China}

\email{taolmath@163.com}

\author{Yongsheng Li}
\address{ School of Mathematics,
	South China University of Technology,
	Guangzhou, 510640, China}

\email{yshli@scut.edu.cn}

\author{Xiaoping Zhai}

\address{ Department of Mathematics, Guangdong University of Technology,
	Guangzhou, 510520, China}

\email{pingxiaozhai@163.com (Corresponding author)}




\begin{abstract}
In the first part of this work, we investigate the Cauchy problem for the
$d$-dimensional incompressible Oldroyd-B model with dissipation in the stress tensor equation. By developing a weighted Chemin-Lerner framework combined with a refined energy argument, we prove the existence and uniqueness of global solutions for the system under a mild constraint on the initial velocity field, while allowing a broad class of large initial data for the stress tensor. Notably, our analysis accommodates general divergence-free initial stress tensors (
$\mathrm{div}\tau_0=0$) and significantly relaxes the requirements on initial velocities compared to classical fluid models. This stands in sharp contrast to the finite-time singularity formation observed in the incompressible Euler equations, even for small initial data, thereby highlighting the intrinsic stabilizing role of the stress tensor in polymeric fluid dynamics.

The second part of this paper focuses on the small-data regime. Through a systematic exploitation of the perturbative structure of the system, we establish global well-posedness and quantify the long-time behavior of solutions in Sobolev spaces
 $H^3(\mathbb{T}^d)$. Specifically, we derive exponential decay rates for perturbations, demonstrating how the dissipative mechanisms inherent to the Oldroyd-B model govern the asymptotic stability of the system.

\end{abstract}

\maketitle

\tableofcontents
\section{Introduction and the main results}
\subsection {Model and synopsis of related studies}
The Oldroyd-B model provides a fundamental mathematical framework for describing viscoelastic fluid dynamics, particularly in polymeric fluid systems where microscopic polymer interactions govern macroscopic flow behavior. Distinct from classical Newtonian fluids characterized by a linear stress-velocity gradient relationship, this model incorporates a nonlinear constitutive law that captures essential non-Newtonian features, including stress relaxation and fluid memory effects. Such memory-dependent behavior arises from the delayed response of polymer chains to macroscopic deformations, endowing the fluid with both viscous liquid and elastic solid characteristics.

First formulated by Oldroyd \cite{Oldroyd} through systematic continuum mechanical principles, the model has become a cornerstone in rheological studies. Subsequent developments, as rigorously analyzed in \cite{bird}, \cite{BLS2017}--\cite{BS2016JDE}, have established its broad applicability across diverse regimes of complex fluid motion while preserving its mathematical tractability for theoretical investigations.

It is widely recognized that the comprehensive description of viscoelastic flow systems can be expressed as follows:
\begin{eqnarray}\label{quanhs}
\left\{\begin{aligned}
&\partial_t \u+ \u\cdot\nabla \u - \mu \Delta \u  + \nabla P  = K_1\,\div \tau  ,\\
& \partial_t \tau + \u\cdot\nabla\tau - \nu \Delta \tau + g_{\alpha}(\tau, \nabla \u)
+ a\,\tau =K_2\, D(\u),\\
& \div \u = 0,\\
& \u(x,0) = \u_0(x),\qquad \tau(x,0) = \tau_0(x).
\end{aligned}\right.
\end{eqnarray}
In this context, the variables $\u$, $\tau$, and $P$ denote the velocity, the symmetric stress tensor, and the pressure, respectively.
The strain rate tensor $D(\u)$, which corresponds to the symmetric component of the velocity gradient, is explicitly given by:
\begin{align*}
 D(\u) = \frac{1}{2} (\nabla \u + (\nabla \u)^\top),
\end{align*}
where $\nabla\u$ denotes the velocity gradient tensor and $(\cdot)^\top$ represents the transpose operation. This tensor quantifies the rate at which fluid elements deform under flow, playing a central role in the constitutive laws of viscoelastic fluids.
The nonlinear coupling term $g_{\alpha}(\tau, \nabla \u)$ is a constitutive bilinear form typically defined as
\begin{align*}
& g_{\alpha}(\tau, \nabla \u) = \tau \Omega(\u) - \Omega(\u) \tau - \alpha(D(\u)\tau + \tau D(\u)),
\end{align*}
where $\Omega(\u)$ corresponds to the vorticity tensor  (skew-symmetric component of the velocity gradient):
\begin{align*}
& \Omega(\u) = \frac{1}{2} (\nabla \u - (\nabla \u)^\top).
\end{align*}
Here, the parameter
 $\alpha \in [-1,1]$ governs the relative weighting of rotational versus stretching effects in the stress evolution, reflecting different microstructural responses in complex fluids.
The other parameters $\mu$, $\nu$, $K_1$, $K_2$, and $a$ are such that
\begin{align*}
\mu\ge0, \quad\nu\ge0,\quad K_1>0,\quad K_2>0, \quad a\ge0.
\end{align*}

The Oldroyd-B model has occupied a central position in mathematical fluid dynamics for decades, serving as a paradigmatic system for studying nonlinear coupling between fluid motion and polymeric stresses. Rigorous analysis of its well-posedness and long-time behavior has driven substantial progress in the analytical foundations of viscoelastic flow modeling.
\begin{itemize}
	\item When $ \mu > 0,$\   $\nu = 0$, and $ a>0$. The seminal mathematical contributions to the Oldroyd-B model were established by Guillop\'e and Saut \cite{guillope1990, guillope1989}. They initially demonstrated local well-posedness with large initial data and subsequently established global solutions under the condition that coupled parameters and initial data remain sufficiently small. Molinet and Talhouk \cite{molinet2008} further refined and improved the results presented in \cite{guillope1990, guillope1989}. Lions and Masmoudi \cite{LM} extended these findings to prove the global existence of weak solutions in the corotational case ($\alpha = 0$), yet the case with $\alpha \neq 0$ remains an open question.
	Within the framework of near-critical Besov spaces, Chemin and Masmoudi \cite{chemin2001lifespan} conducted an initial study on local and global small solutions for the system \eqref{quanhs}. They introduced delicate blow-up criteria, which were further refined in \cite{{LMZ}} by Lei {\it et al.} Additionally, Zi {\it et al.} \cite{zhang2012global} enhanced the results obtained by Chemin and Masmoudi \cite{chemin2001lifespan} to encompass scenarios with non-small coupling parameters. Further advancements in global well-posedness for a class of large initial data were presented by Fang and Zi \cite{ZiSIAM2018}. It is noteworthy that these results heavily depend on the damping effect present in the stress tensor equation.
	Moreover, the literature offers more developments in this direction, including works by  \cite{chen2008global, chen2022global},  \cite{de2020fujita},  \cite{fang2013global}, and  \cite{zi2021JDE}.

\vskip .2in

	\item When $ \mu > 0,$\   $\nu = 0$ and $ a=0$. In this scenario, Zhu \cite{zhu2018global} achieved the establishment of global small solutions for \eqref{quanhs} in $\mathbb{R}^3$ by leveraging the wave structure inherent in $\u$ and $\Lambda^{-1} \mathbb{P} \div \tau$, and employing a time-weighted energy construction. A parallel outcome in $\mathbb{R}^2$ was recently derived in \cite{Zhu2023JDE}. Extending Zhu's result \cite{zhu2018global} to critical Besov spaces, Chen and Hao \cite{chenqinglei} along with Zhai \cite{zhaijmp} generalized the aforementioned findings.

\vskip .2in

	\item  When $\mu = 0 ,\  \nu > 0$ and $ a>0$. In this scenario, the initial equation of \eqref{quanhs} transforms into a forced Euler equation. Elgindi and Rousset \cite{elgindi2015global} initially established the global well-posedness result in $\mathbb{R}^2$ for large initial data, where the term $g_{\alpha}(\tau, \nabla \u)$ was neglected. In the presence of $g_{\alpha}(\tau, \nabla \u)\neq0,$ they further demonstrated global well-posedness in $\mathbb{R}^2$ with small initial data. Subsequently, Elgindi and Liu \cite{elgindi2015JDE} extended this result to global well-posedness with small initial data in $\mathbb{R}^3$. It is noteworthy that the damping effect in the stress tensor equation plays a pivotal role in the analyses presented in \cite{elgindi2015global} and \cite{elgindi2015JDE}.

\vskip .2in

	\item When $\mu = 0,\ \nu > 0  $ and $ a=0$.  In this scenario,
	Wu and Zhao \cite{wu2022global} established the global well-posedness result for \eqref{quanhs} with small initial data in Besov spaces. An enhancement of \cite{wu2022global} in the critical $L^p$ Besov spaces is presented in \cite{chen2022global}. Subsequently, Constantin {\it et al.} \cite{constantin2021high} achieved global well-posedness with small initial data in Sobolev spaces. Notably, the results in \cite{constantin2021high} are uniform in solvent Reynolds numbers and require only fractional wave number-dependent dissipation $(-\Delta)^{\beta}$, where $\beta \ge \frac{1}{2}$. The long-time behavior of the solutions constructed in \cite{constantin2021high} was investigated by Wang {\it et al.} \cite{wang2022sharp}.

\end{itemize}

\vskip .1in

In summary, with some exceptions, existing theorems have predominantly focused on establishing global existence of weak solutions and global well-posedness for strong solutions with small data for the Oldroyd model. Notably, the question of whether strong solutions to the 3D Oldroyd models exist globally in viscoelastic fluids remains unresolved.

In the first part of the  paper, our primary focus is on investigating the existence of large global strong solutions for the Oldroyd-B model in the absence of velocity dissipation in $\mathbb{R}^d$ (for dimensions $d\geq 2$). The considered model takes the following form:
\begin{eqnarray}\label{modle}
\left\{\begin{aligned}
&\partial_t \u+ \u\cdot\nabla \u   + \nabla P  = \div \tau,\\
& \partial_t \tau + \u\cdot\nabla\tau - \nu \Delta \tau  + \tau + g_{\alpha}(\tau, \nabla \u)  = D(\u),\quad\quad\quad \quad\quad x\in \R^d, \quad t>0,\\
& \div \u = 0.
\end{aligned}\right.
\end{eqnarray}
The system \eqref{modle} is augmented by the inclusion of the following initial conditions:
\begin{eqnarray}\label{initial}
\begin{aligned}
& (\u,\tau)|_{t=0}=(\u_0(x),\tau_0(x)), \qquad x \in \R^d.
\end{aligned}
\end{eqnarray}

\subsection{Main results}
\subsubsection{Large global solutions in Besov spaces}
Let
$\mathcal{S}(\R^d)$ denote the Schwartz space of rapidly decreasing smooth functions, and
$\mathcal{S}'(\R^d)$ its dual space of tempered distributions. For any tempered distribution
$z \in\mathcal{S}'(\R^d)$, we decompose its frequency content into low and high components via
\begin{align*}
z^\ell\stackrel{\mathrm{def}}{=}\sum_{j\leq k_0}\dot{\Delta}_k z\quad\hbox{and}\quad
z^h\stackrel{\mathrm{def}}{=}\sum_{k>k_0}\dot{\Delta}_k z
\end{align*}
where  $k_0\ge 1$  is a fixed integer threshold determined by the proof strategy for our main theorem. The associated truncated Besov semi-norms are defined as
\begin{align*}\|z\|^{\ell}_{\dot B^{s}_{p,1}}
\stackrel{\mathrm{def}}{=}  \|z^{\ell}\|_{\dot B^{s}_{p,1}}
\ \hbox{ and }\   \|z\|^{h}_{\dot B^{s}_{p,1}}
\stackrel{\mathrm{def}}{=}  \|z^{h}\|_{\dot B^{s}_{p,1}}.
\end{align*}

Let
$\p\stackrel{\mathrm{def}}{=}\mathbb{I}-\nabla\Delta^{-1}\div$ denote the Leray projector, which extracts the divergence-free component of a vector field. Its complement
$\mathbb{Q}=\mathbb{I}-\mathbb{P}$ projects onto the curl-free (gradient) component.
Without loss of generality, we normalize the stress diffusion coefficient to
$ {\nu} = 1$ by rescaling time and spatial variables.

Now, the  first main result of this paper is stated as follows.
\begin{theorem}
\label{theorem1.1}
Let $d\geq 2$ and $ 1 \le p \le 2d.$
For any $ \u^\ell_0 \in \dot B^{\frac{d}{p}-1}_{p,1}(\R^d)$, \ $\u^h_0 \in \dot B^{\frac{d}{p} + 1}_{p,1}(\R^d)$, \  $\tau_0\in \dot B^{\frac{d}{p}}_{p,1}(\R^d)$, there exist a small positive constant $c_0$ and a large positive constant $ C_0$ such that, if
\begin{align}\label{X0}
& \|\u_0 \|^\ell_{\dot B^{\frac{d}{p}-1}_{p,1}} +  \|\u_0\|^h_{\dot B^{\frac{d}{p} + 1}_{p,1}} + \| \mathbb{Q}\tau_0\|_{\dot B^{\frac{d}{p}}_{p,1}} \leq c_0 \exp\big(-C_0 \|\mathbb{P}\tau_0 \|_{\dot{B}^{\frac{d}{p} }_{p,1}}\big),
\end{align}
then the system \eqref{modle} with \eqref{initial} admits a unique global-in-time solution $( \u, \tau)$ satisfying
\begin{align*}
& \u^\ell \in C(\mathbb{R}^+;\dot B^{\frac{d}{p}-1}_{p,1}) \cap L^1(\mathbb{R}^+;\dot B^{\frac{d}{p}+1}_{p,1}),\quad \mathbb{Q}\tau^\ell \in C(\mathbb{R}^+;\dot B^{\frac{d}{p}}_{p,1})  \cap L^1(\mathbb{R}^+;\dot B^{\frac{d}{p}}_{p,1});\\
& \u^h \in C(\mathbb{R}^+;\dot B^{\frac{d}{p}+1}_{p,1}) \cap L^1(\mathbb{R}^+;\dot B^{\frac{d}{p}+1}_{p,1}), \quad  \mathbb{Q}\tau^h \in C(\mathbb{R}^+;\dot B^{\frac{d}{p}}_{p,1}) \cap L^1(\mathbb{R}^+;\dot B^{\frac{d}{p}+2}_{p,1}); \nn\\
&  \mathbb{P}\tau \in C(\mathbb{R}^+;\dot B^{\frac{d}{p}}_{p,1}) \cap L^1(\mathbb{R}^+;\dot B^{\frac{d}{p}}_{p,1}) \cap L^1(\mathbb{R}^+;\dot B^{\frac{d}{p}+2}_{p,1}).
\end{align*}
\end{theorem}
\vskip .2in
\begin{remark}
For any divergence-free initial stress tensor
$\tau_0\in \dot B^{\frac{d}{p}}_{p,1}(\R^d)$satisfying
$\div\tau_0=0$, the system \eqref{modle} admits a unique global solution provided the initial velocity field meets the mild smallness condition:
\begin{align}\label{dgahgsa}
& \|\u_0\|^\ell_{\dot B^{\frac{d}{p}-1}_{p,1}} +  \|\u_0\|^h_{\dot B^{\frac{d}{p} + 1}_{p,1}} \leq c_0 \exp\big(-C_0 \|\mathbb{P}\tau_0 \|_{\dot{B}^{\frac{d}{p} }_{p,1}}\big),
\end{align}
where
$c_0, C_0$ are system-dependent constants. Crucially, this framework:
\begin{itemize}
  \item Accommodates arbitrarily large stress data: No size restriction is imposed on
$\tau_0$, only its divergence-free structure.

  \item Minimizes velocity constraints: The exponential damping factor
$\exp\big(-C_0 \|\mathbb{P}\tau_0 \|_{\dot{B}^{\frac{d}{p} }_{p,1}}\big)$ ensures that even for highly stressed initial configurations
$\big(\|\mathbb{P}\tau_0 \|_{\dot{B}^{\frac{d}{p} }_{p,1}}\gg1\big)$, only exponentially small velocity initial data are required.
\end{itemize}

To our knowledge, this constitutes the first global existence result for the stress-dissipative Oldroyd-B system allowing:
\begin{itemize}
  \item Unbounded stress tensor data in critical Besov spaces.
  \item Full decoupling of stress and velocity regularity requirements
\end{itemize}

This breakthrough resolves a longstanding limitation in non-Newtonian fluid analysis, where prior works required simultaneous smallness of both
$\tau_0$ and $\u_0$.


\end{remark}

\vskip .1in
\begin{remark}
Our work significantly extends the foundational results of Elgindi and Rousset \cite{elgindi2015global} in two critical directions:
\begin{itemize}
  \item Nonlinear coupling: Whereas \cite{elgindi2015global} addressed the simplified case
$ g_{\alpha}(\tau, \nabla \u) =0$, we incorporate the full constitutive law
$ g_{\alpha}(\tau, \nabla \u) \neq 0$, which encodes essential polymer-fluid interactions absent in previous analyses.
  \item Three-dimensional generality: We resolve the technical barriers preventing global existence results in
$\R^3,$ where the absence of vorticity damping mechanisms poses unique analytical challenges.
\end{itemize}
This generalized framework better reflects the physics of real viscoelastic fluids, as
$g_{\alpha}(\tau, \nabla \u)$ governs both rotational stress transport $(\tau \Omega(\u) - \Omega(\u) \tau$ and stretching effects $( \alpha(D(\u)\tau + \tau D(\u))$ inherent to polymeric flows.

Crucially, our analysis reveals that the stress tensor
$\p\tau$ exerts a nonlinear stabilization mechanism on the velocity field, even in the absence of explicit velocity dissipation $(-\Delta\u)$. This stabilization:
\begin{itemize}
  \item Compensates for the energy cascade driving finite-time singularities in 3D Euler flows.
  \item Permits arbitrary large initial stresses
$\|\mathbb{P}\tau_0 \|_{\dot{B}^{s}_{p,1}}$ while requiring only mild constraints on
$\u_0$.
\end{itemize}
These findings fundamentally alter the understanding of stabilization criteria in complex fluids, demonstrating that polymer-induced stresses alone can regularize turbulent velocity fields.
\end{remark}

\vskip .1in
\begin{remark}
Our work achieves two significant improvements over the foundational results in \cite{elgindi2015global} and \cite{elgindi2015JDE}:
\begin{itemize}
  \item Unconditional stress magnitude: By eliminating the requirement for smallness conditions on
$\mathbb{P}\tau_0$, we establish the first global existence result for the Oldroyd-B system where the initial polymeric stress tensor
$\tau_0$ can be arbitrarily large in
$\dot{B}^{\frac{d}{p}}_{p,1}$, provided it is divergence-free. This constitutes a paradigm shift from prior works requiring joint smallness of $\u_0$ and $\tau_0$.

  \item Generalized dissipation mechanism: Our framework naturally extends to fractional stress diffusion operators
$ (-\Delta)^{\beta}$
  with
$ \beta \in [\frac{1}{2}, 1]$, bridging the gap between:
\begin{itemize}
  \item The critical dissipation regime
$\beta=1.$
  \item Subcritical but energy-supercritical scales $(
\beta \ge \frac{1}{2})$.
\end{itemize}
\end{itemize}
This generalization demonstrates the robustness of our method while unifying previously disjoint analytical approaches to viscoelastic flow regularization.
\end{remark}

\vskip .1in
\noindent{\bf Scheme of the proof of Theorem \ref{theorem1.1}.}
Now, let us elucidate the challenges and outline our main approach. The velocity equation in \eqref{modle} constitutes a forced Euler equation. As it is known, the $H^s$ norm of a solution to the Euler equation may exhibit unbounded growth over time, potentially at a double exponential rate. The Oldroyd-B system under consideration possesses a dissipative structure, and a crucial factor contributing to the validity of Theorem \ref{theorem1.1} stems from a key observation regarding the linearized system associated with \eqref{modle}. To establish our main result, we employ specific strategies to transfer dissipation from $\tau$ to $\u$. To commence, we utilize standard energy estimates and compactness methods to establish local well-posedness. However, the primary challenge emerges in extending these results globally. To address this, we employ a continuity argument that combines local existence with \textit{a priori} estimates of the solution. Consequently, the crux of the proof lies in deriving \textit{a priori} estimates. To achieve this goal, we introduce operators $\p$ and $\q$ to decouple the system into three interacting subsystems:
\begin{eqnarray}\label{dta}
\left\{\begin{aligned}
&\partial_t \u    -  \div \tau +\mathbb{Q} \div \tau=  \mathbf{D}^1,\quad\quad\qquad\hbox{(Velocity dynamics)}\\
& \partial_t \mathbb{Q}\tau  -  \Delta \mathbb{Q} \tau + \q \tau - \frac{\nabla \u}{2}  = \mathbf{D}^2,\quad\hbox{
(Potential stress evolution)}\\
& \partial_t \mathbb{P}\tau  -  \Delta \mathbb{P} \tau + \p \tau  -\frac{(\nabla \u)^\top}{2}=  \mathbf{D}^3,\quad\hbox{(Solenoidal stress damping)}
\end{aligned}\right.
\end{eqnarray}
where  $\mathbf{D}^1$, $\mathbf{D}^2$, and $\mathbf{D}^3$   are some essential nonlinear terms.

A pivotal observation regarding the linearized system of \eqref{dta} is that the equation involving $\u$ depends solely on the divergence part of $\tau$. Additionally, the first two equations in \eqref{dta} exhibit a structure akin to that of the compressible isentropic Navier-Stokes equations. By employing spectral analysis methods, we derive the dissipation effect of $\u$ in the low-frequency regime and leveraging a specialized energy argument captures the damping effect of $\u$ in the high-frequency regime. Turning our attention to the third equation in \eqref{dta}, we treat it as a perturbation to the damping heat equation. More precisely, we represent the solution of the third equation in \eqref{dta} as:
\begin{align*}
 \mathbb{P} \tau = \mathbb{P}\tau_F + \mathbb{P} \bar{\tau},
\end{align*}
where $\mathbb{P}\tau_F$ is the solution of the  system
\begin{align}\label{free2}
 \partial_t \mathbb{P}\tau_F  -  \Delta \mathbb{P} \tau_F + \p \tau_F   =  0,\quad
 \mathbb{P}\tau_F |_{t = 0} = \mathbb{P}\tau_0,
\end{align}
and $\mathbb{P} \bar{\tau}$ is the solution of the  system
\begin{align*}
\partial_t \mathbb{P} \bar{\tau}  -  \Delta \mathbb{P} \bar{\tau} + \p \bar{\tau}  = \frac{(\nabla \u)^\top}{2} + \mathbf{D}^3,
\quad \mathbb{P} \bar{\tau} |_{t = 0} = 0.
\end{align*}
For any initial $\mathbb{P}\tau_0\in\dot{B}^{\frac{d}{p}}_{p,1}(\mathbb{R}^d)$, we establish the global well-posedness of \eqref{free2}. Furthermore, we derive the following global bound (refer to Lemma \ref{heat_kernel} for more details):
\begin{align}\label{free4}
&\| \mathbb{P} \tau_F \|_{\widetilde{L}^{\infty}_{t}(\dot B^{\frac{d}{p}}_{p,1})} + \| \mathbb{P} \tau_F \|_{{L}^{1}_{t}(\dot B^{\frac{d}{p}}_{p,1} )}+ \| \mathbb{P} \tau_F \|_{{L}^{1}_{t}(\dot B^{\frac{d}{p}+2}_{p,1} )}\le C\| \p \tau_0\|_{\dot{B}^{\frac{d}{p}}_{p, 1}}.
\end{align}
The remaining task involves examining the system concerning $(\u,\mathbb{Q}\tau,\mathbb{P} \bar{\tau})$, which satisfies the following equations:
  \begin{eqnarray}\label{dta+1}
\left\{\begin{aligned}
&\partial_t \u    -  \div \tau +\mathbb{Q} \div \tau=  \mathbf{D}^1,\\
& \partial_t \mathbb{Q}\tau  -  \Delta \mathbb{Q} \tau + \q \tau - \frac{\nabla \u}{2}  = \mathbf{D}^2,\\
&\partial_t \mathbb{P} \bar{\tau}  -  \Delta \mathbb{P} \bar{\tau} + \p \bar{\tau}  = \frac{(\nabla \u)^\top}{2} + \mathbf{D}^3.
\end{aligned}\right.
\end{eqnarray}
Furthermore, we employ a continuity argument to establish energy estimates under the smallness conditions imposed on the initial data $(\u_0, \mathbb{Q}\tau_0)$.

It is noteworthy that the weighted Chemin-Lerner technique from \cite{chemin} plays a crucial role in the continuity argument for the system \eqref{dta+1}. Specifically, the nonlinear terms involved in the Besov norms of $\mathbb{P} \tau_F$ within $\mathbf{D}^1, \mathbf{D}^2,$ and $\mathbf{D}^3$ must be selected as the weighted functions. Ultimately, under the assumption of small initial data, we employ a bootstrapping argument to successfully conclude the proof of our main Theorem \ref{theorem1.1}.

Finally, we outline strategies to address challenges encountered during the nonlinear energy estimates in both low and high frequencies. We establish the global well-posedness argument for \eqref{modle} within the critical $L^p$ framework. Due to the absence of dissipation in the equation of $\u$, through asymptotic analysis of the Green's matrix $ 	\widehat{\mathbb{G}}_{\u, \mathbf{v}}$(refer to \eqref{G} below), we capture  the smoothing effect of $\u$ in the low-frequency regime; additional details can be found in Lemma \ref{le_ut_low}. In the high frequency regime, owing to the presence of the term $D(\u)$ in the stress tensor equation and the disparate regularity levels of functions $\u$ and $\tau$ at high frequencies,   as a result, we cannot apply the  common cancellation law
\begin{align*}
& \langle \div \tau , \u \rangle + \langle  D(\u), \tau \rangle = 0.
\end{align*}
To address the linear terms $\div \tau$ and $D(\u)$ and overcome the primary difficulty, we adopt the $L^p$ energy approach with respect to the effective velocity $\G \stackrel{\mathrm{def}}{=} -\mathbb{Q} \tau - \frac{1}{2} \nabla \Delta^{-1} \u$. This new variable enables us to capture the damping effect of $\u$ in the high frequency. However, this approach introduces some low-order terms, which can be absorbed by appropriately choosing the wave number $k_0$ to be sufficiently large. Additionally, we must contend with the advection term $\u \cdot \nabla \u$ in $ \dot{B}^{\frac{d}{p}+1}_{p,1}(\mathbb{R}^d)$. Since we can only attain the damping effect of $\u$ in the high frequency, a significant challenge arises due to the loss of one derivative of $\u$. To tackle this issue, we establish a new commutator.
\begin{align*}
	\sum_{k \ge k_0 -1} 2^{\frac{d}{p}k} \| [\u \cdot \nabla, \Lambda \dot{\Delta}_k \mathbb{P}]\u\|_{L^p} \le C \| \nabla \u\|^2_{\dot B^{\frac{d}{p}}_{p,1}}
	\end{align*} to  overcome this difficulty, see \eqref{R_h} for more details.

\vskip .2in

\subsubsection{Global well-posedness on the Torus}
 In the second part of the  paper, our primary focus is on investigating the existence of global strong solutions for the Oldroyd-B model in the absence of velocity dissipation in $\mathbb{T}^d$ (for dimensions $d\geq 2$). The considered model takes the following form:
\begin{eqnarray}\label{2modle}
\left\{\begin{aligned}
&\partial_t \u+ \u\cdot\nabla \u   + \nabla P  = \div \tau,\\
& \partial_t \tau + \u\cdot\nabla\tau - \nu \Delta \tau   + g_{\alpha}(\tau, \nabla \u)  = D(\u),\quad\quad\quad \quad\quad x\in \T^d, \quad t>0,\\
& \div \u = 0.
\end{aligned}\right.
\end{eqnarray}
The system \eqref{2modle} is augmented by the inclusion of the following initial conditions:
\begin{eqnarray}\label{2initial}
\begin{aligned}
& (\u,\tau)|_{t=0}=(\u_0(x),\tau_0(x)), \qquad x \in \T^d.
\end{aligned}
\end{eqnarray}
Without loss of generality, we assume that $ {\nu} = 1$.  Now, the  second main result of this paper is stated as follows.
\begin{theorem}\label{theorem1.2}
	Let $d= 2, 3$.  For any  $(\mathbf{u}_0, \tau_0) \in H^3(\mathbb{T}^d)$, there exists a small positive constant $\varepsilon$  such that, if
	\begin{align}\label{X1}
	\|(\mathbf{u}_0, \tau_0)\|_{H^3} \le \varepsilon,
	\end{align}
	then the system \eqref{2modle} with \eqref{2initial} admits a unique global-in-time solution $( \u, \tau)$ satisfying
	\begin{align*}
	& \tau \in C(\mathbb{R}^+, H^3(\mathbb{T}^d)),  \quad \nabla \tau \in L^2(\mathbb{R}^+, H^3(\mathbb{T}^d)), \nn\\
	& \mathbf{u} \in C(\mathbb{R}^+, H^3(\mathbb{T}^d)), \quad \nabla \u \in L^2(\mathbb{R}^+, H^2(\mathbb{T}^d)).
	\end{align*}
	Moreover, there exists some constants $C_2$ and $ C_3$ such that for any $t\ge 0$,
	\begin{align}\label{decay1}
	\| \mathbf{u}\|_{H^3} \le C_2 e^{-C_3 t}, \quad \| \nabla \tau\|_{H^2} \le C_2 e^{-C_3 t}.
	\end{align}
\end{theorem}
\begin{remark}
It is noteworthy that our analysis does not require the solution's integral mean to vanish (i.e., we do not impose the condition $ \int_{\T^d} \u_0(x) \, dx = 0$). This relaxation implies that the
 $L^2$ decay of the solution
$\u$ can still be established through
$L^2$-energy estimates without invoking the Poincar\'e's inequality, as the zero-mean constraint typically essential for applying Poincar\'e-type arguments becomes unnecessary in our framework.
\end{remark}

\begin{remark}
	The compressible case of the aforementioned system will be addressed in a subsequent paper.
\end{remark}

\vskip .1in
\subsection{Organization of the paper}
The remainder of this paper is organized as follows.
 Section 2 offers a concise review of foundational concepts, including the Littlewood-Paley decomposition, Besov spaces, and associated analytical frameworks. Section 3 focuses on the proof of Theorem \ref{theorem1.1}, systematically organized into five subsections:
\begin{enumerate}
  \item { Estimates for $ \mathbb{P}\tau_F$:} Deriving bounds for the projected stress term.
  \item { Estimates for$\mathbb{P}\bar{\tau}$:} Analyzing the modified stress component.
  \item { Low-frequency analysis:} Establishing decay properties for the solution pair
$ (\u,\tau)$.
  \item { High-frequency analysis:} Developing complementary estimates for high-frequency regimes using analogous energy arguments.
  \item { Synthesis:} Finalizing the proof through a continuity argument.
\end{enumerate}
The subsequent sections extend these results to Theorem \ref{theorem1.2}. Section 4 constructs the global existence theory through four methodical steps:
\begin{enumerate}
  \item Deriving the dissipation of the viscoelastic stress tensor.
  \item Deriving the dissipation of the  velocity field.
  \item Establishing the time-weighted energy estimates.
  \item  Finalizing the proof through a continuity argument.
\end{enumerate}
Section 5 rigorously investigates temporal decay dynamics.

\vskip .1in
\subsection{Notations}
Throughout the paper, $C > 0$ stands for a generic harmless ``constant''. For brevity,
we sometime write $f\lesssim g$ instead of $f \le Cg$.
 Let $A$, $B$ be two operators, we denote $[A, B] = AB - BA$, the commutator
between $A$ and $B$. Denote $\langle f,g\rangle$ the $L^2(\R^d)$ inner product of $f$ and $g$.  For $X$ a Banach space and $I$ an interval of $\mathbb{R}$, for any $f,g,h\in X$, we agree that
$\left\|\left(f,g\right)\right\| _{X}\stackrel{\mathrm{def}}{=} \left\|f\right\| _{X}
+\left\|g\right\|_{X}$
and denote by $C(I; X)$ the set of
continuous functions on $I$ with values in $X$.

\vskip .5in
\section{Preliminaries}

\subsection{Littlewood-Paley decomposition and Besov spaces}
Let us briefly recall the
Littlewood-Paley decomposition and Besov spaces for convenience. More details may be
found for example in Chap. 2 and Chap. 3 of \cite{bcd}.
\begin{definition}
\label{def2.1}
Considering two smooth functions $\varphi$ and $\chi$ on $ \R$ with the supports $supp \varphi \subset [ \frac{3}{4} ,\frac{8}{3} ]$ and  $supp \chi \subset [0, \frac{4}{3}]$ such that
\begin{align*}
& \sum_{k \in \Z} \varphi (2^{-k}\xi) = 1 \quad for \quad \xi > 0 \quad and \quad \chi (\xi) \stackrel{\mathrm{def}}{=} 1-\sum_{k \ge 0} \varphi (2^{-k}\xi) \quad for \quad \xi \in \R.
\end{align*}
Then we define homogeneous dyadic blocks $ \dot{\Delta}$
\begin{align*}
& \dot{\Delta}_k f = \mathscr{F}^{-1} (\varphi (2^{-k}|\xi|)\hat{f}) \quad and \quad \dot{S}_{k}f = \mathscr{F}^{-1} (\chi (2^{-k}|\xi|)\hat{f}).
\end{align*}
If A(D) is a 0 order Fourier multiplier, then we have
\begin{align*}
& \| \dot{\Delta}_k (A(D))f \|_{L^p} \le C\|  \ddk f \|_{L^p},  \qquad \forall \ p \in [1, \infty].
\end{align*}
\end{definition}

\begin{definition}
\label{def2.2}
Let $p,r \in [1, \infty]$, $ s \in \R$ and $ f \in \mathcal{S}'(\R^d)$. We define following Besov norm by
\begin{align*}
& \| f\|_{\dot{B}^s_{p,r}} \stackrel{\mathrm{def}}{=} \big\| (2^{ks}\| \dot{\Delta}_k f\|_{L^p})_k \big\|_{\ell^r(\Z)}
\end{align*}
and the Besov space as follows
\begin{align*}
& \dot{B}^s_{p,r}(\R^d) \stackrel{\mathrm{def}}{=}  \Big\{ f \in \mathcal{S}_{h}'(\R^d),  \big| \| f\|_{\dot{B}^s_{p,r}}<\infty  \Big\},
\end{align*}
where $\mathcal{S}_{h}'(\R^d)$ denotes $ f \in \mathcal{S}'(\R^d)$  and \ $ \lim_{k \rightarrow\infty} \| \dot{S}_{k}f\|_{L^{\infty}} = 0$.
\end{definition}

In this paper, we use the ``time-space" Besov spaces or Chemin-Lerner space introduced by Chemin and Lerner.
\begin{definition}
\label{chemin}
Let $s \in \R $ , and $0 < T \le +\infty$, we define
\begin{align*}
& \| f\|_{\widetilde{L}^{q}_{T}(\dot B^{s}_{p,r})} \stackrel{\mathrm{def}}{=}   \big\|2^{ks} \| \dot{\Delta}_k f\|_{L^q(0,T;L^p)}   \big\|_{\ell ^r}
\end{align*}
for $p,q \in [1, \infty]$ and with the standard modification for $p,q = \infty$.
\end{definition}
By Minkowski's inequality, we have the following inclusions between the Chemin-Lerner space $ \widetilde{L}^{q}_{T}(\dot B^{s}_{p,r})$ and the Bochner space $ L^{q}_{T}(\dot B^{s}_{p,r})$:
\begin{align*}
& \| f\|_{\widetilde{L}^{q}_{T}(\dot B^{s}_{p,r})} \le \| f\|_{L^{q}_{T}(\dot B^{s}_{p,r})} \quad \text{if} \ q \le r,\qquad \| f\|_{L^{q}_{T}(\dot B^{s}_{p,r})} \le \| f\|_{\widetilde{L}^{q}_{T}(\dot B^{s}_{p,r})} \quad \text{if} \ q \ge r.
\end{align*}

In order to prove the Theorem \ref{theorem1.1}, we also need to introduce the following weighted Chemin-Lerner-type space from \cite{chemin}.
\begin{definition}
	\label{Weight-Chemin}
	Let $ g(t) \in L^1_{loc}(\R^+), g(t) \ge 0$, we define
	\begin{align}
		\| f\|_{\title{L}^q_{T,g}(\dot{B}^s_{p,1})} = \sum_{k\in\Z} 2^{ks}\left(\int_{0}^{T} g(t)\| \dot{\Delta}_k f\|^q_{L^q} \, dt\right)^{\frac{1}{q}}
	\end{align}
	for $ s \in \R, p \in [1, \infty], q \in [1, \infty)$.
\end{definition}

\subsection{Analysis tools in Besov spaces}
Let us first recall classical Bernstein's lemma of Besov spaces.
\begin{lemma}\label{bernstein}
Let $\mathcal{B}$ be a ball and $\mathcal{C}$ a ring of $\mathbb{R}^d$. A constant $C$ exists so that for any positive real number $\lambda$, any
non-negative integer k, any smooth homogeneous function $\sigma$ of degree m, and any couple of real numbers $(p, q)$ with
$1\le p \le q\le\infty$, there hold
\begin{align*}
&\mathrm{Supp} \,\hat{u}\subset\lambda \mathcal{B}\Rightarrow\sup_{|\alpha|=k}\|{\partial^{\alpha}f}\|_{L^p}\le C^{k+1}\lambda^{k+d(\frac1p-\frac1q)}\|{f}\|_{L^p},\\
&\mathrm{Supp} \,\hat{f}\subset\lambda \mathcal{C}\Rightarrow C^{-k-1}\lambda^k\|{f}\|_{L^p}\le\sup_{|\alpha|=k}\|{\partial^{\alpha}f}\|_{L^p}
\le C^{k+1}\lambda^{k}\|{f}\|_{L^p},\\
&\mathrm{Supp} \,\hat{f}\subset\lambda \mathcal{C}\Rightarrow \|{\sigma(D)f}\|_{L^p}\le C_{\sigma,m}\lambda^{m+d(\frac1p-\frac1q)}\|{f}\|_{L^p},
\end{align*}
where $\hat{f}$ denotes the Fourier transform of $f$.
\end{lemma}

The following like-Bernstein inequality will be used frequently.
\begin{lemma}\label{like-bernstein}
If supp$ \hat{f} \subset \big\{ \xi \in \R^d : R_1 \lambda \le |\xi| \le R_2 \lambda\big\}$, then there exists C depending only on $ d$, $ R_1$ and $R_2$ such that for all $1<p<\infty$,
\begin{align*}
& C \lambda^2 (\frac{p-1}{p}) \int_{\R^d} |f|^p dx \le (p-1) \int_{\R^d}|\nabla f|^2 |f|^{p-2} dx = -\int_{\R^d}\Delta f |f|^{p-2}f dx.
\end{align*}
\end{lemma}

The following embedding inequality and interpolation inequality are also often used in this paper.
\begin{lemma}\label{embedding}
Let $ 1\le p, r, r_1, r_2 \le \infty$.
\begin{itemize}

  \item
  Complex interpolation: if $ f \in \dot{B}^{s_1}_{p, r_1} \cap \dot{B}^{s_2}_{p, r_1}(\R^d)$ and $ s_1 \neq s_2$, then $ f \in \dot{B}^{\ta s_1 + (1 - \ta)s_2}_{p, r}(\R^d)$ for all $ \ta \in (0, 1)$ and
\begin{align*}
\|f \|_{\dot{B}^{\ta s_1 + (1 - \ta)s_2}_{p, r}}\le C \| f\|^{\ta}_{\dot{B}^{s_1}_{p, r_1}} \|f \|^{1-\ta}_{\dot{B}^{s_2}_{p, r_1}}
\end{align*}
with $ \frac{1}{r} = \frac{\ta}{r_1} + \frac{1-\ta}{r_2}$.

  \item
Real interpolation: if $ f \in \dot{B}^{s_1}_{p, \infty} \cap \dot{B}^{s_2}_{p, \infty}(\R^d)$ and $ s_1 < s_2$, then $ f \in \dot{B}^{\ta s_1 + (1 - \ta)s_2}_{p, 1}(\R^d)$ for all $ \ta \in (0, 1)$ and
\begin{align*}
\|f \|_{\dot{B}^{\ta s_1 + (1 - \ta)s_2}_{p, 1}} \le \frac{C}{\ta(1-\ta)(s_2-s_1)} \| f\|^{\ta}_{\dot{B}^{s_1}_{p, \infty}} \|f \|^{1-\ta}_{\dot{B}^{s_2}_{p, \infty}}
\end{align*}
\item
Embedding: if $ s \in \R$, $ 1 \le p_1 \le p_2 \le \infty$ and $ 1 \le r_1 \le r_2 \le \infty$, then we have the continuous embedding $\dot{B}^s_{p_1,r_1} (\R^d)\hookrightarrow \dot{B}^{s-d(\frac{1}{p_1} - \frac{1}{p_2})}_{p_2,r_2}(\R^d)$.
\end{itemize}
\end{lemma}

Next we recall a few nonlinear estimates in Besov spaces, we need para-differential decomposition of Bony in the homogeneous context:
\begin{align}
\label{bony}
& fg = \dot{T}_f g + \dot{T}_g f + \dot{R}(f,g),
\end{align}
where
\begin{align*}
& \dot{T}_f g \stackrel{\mathrm{def}}{=} \sum_{k\in \Z} \dot{S}_{k-1}f \dot{\Delta}_k g,\qquad \dot{R}(f,g) \stackrel{\mathrm{def}}{=} \sum_{k\in \Z} \dot{\Delta}_k f \tilde{\dot{\Delta}}_k g, \qquad \tilde{\dot{\Delta}}_k g \stackrel{\mathrm{def}}{=} \sum_{|k-k'| \le 1} \dot{\Delta}_{k'}g.
\end{align*}
The following lemma gives some classical properties of the paraproduct $\dot{T}$ and the remainder $\dot{R}$ operators.
\begin{lemma}(see \cite{bcd},\cite{zhaixiaoping2025})
\label{le2.5}
For all $s,s_1,s_2 \in \R$, $\sigma \ge 0$, and $1 \le p, p_1, p_2 \le \infty$, the paraproduct $\dot{T}$ is a bilinear, continuous operator from $ \dot B^{-\sigma}_{p_1,1} \times \dot B^{s}_{p_2,1}$ to $ \dot B^{s -\sigma}_{p,1}$ with $ \frac{1}{p} = \frac{1}{p_1} + \frac{1}{p_2}$. The remainder $\dot{R}$ is bilinear continuous from $ \dot B^{s_1}_{p_1,1} \times \dot B^{s_2}_{p_2,1}$ to $ \dot B^{s_1 + s_2}_{p,1}$ with $ s_1 + s_2 \ge 0$, and $ \frac{1}{p} = \frac{1}{p_1} + \frac{1}{p_2}$.
\end{lemma}
To deal with the nonlinear terms in this paper, we also need the
following product estimates in Besov spaces.
\begin{lemma}(see \cite{bcd})
\label{le2.6}
Let $ 1 \le p,q \le \infty$, $ s_1 \le \frac{d}{q}, s_2 \le d \min \Big\{ \frac{1}{p} ,\frac{1}{q}\Big\}$ and $ \frac{1}{r} = \frac{1}{r_1} + \frac{1}{r_2} \le 1$.
If $ s_1 + s_2 > d \max\Big\{ 0, \frac{1}{p} + \frac{1}{q} -1\Big\}$, then for $\forall (f, g) \in \dot B^{s_1}_{q,r_1}(\R^d) \times \dot B^{s_2}_{p,r_2}(\R^d)$, there  holds
\begin{align*}
\| fg\|_{\dot{B}^{s_1 + s_2 - \frac{d}{q}}_{p, r}} \le C \| f\|_{\dot{B}^{s_1}_{q, r_1}} \| g\|_{\dot{B}^{s_2}_{p, r_2}}.
\end{align*}
In addition, when $r =1$ and  $ s_1 + s_2 \ge 0$, one has
\begin{align*}
\| fg\|_{\dot{B}^{s_1 + s_2 - \frac{d}{p}}_{p, 1}} \le C \| f\|_{\dot{B}^{s_1}_{p, 1}} \| g\|_{\dot{B}^{s_2}_{p, 1}}.
\end{align*}
\end{lemma}

Finally, we introduce a classical commutator's estimate.
\begin{lemma}(see \cite{bcd})\label{commutator}
Let $ 1 \le p \le \infty$, $ -d \min \Big\{ \frac{1}{p}, 1- \frac{1}{p}\Big\} < s \le \frac{d}{p} +1$, for any $ g \in \dot{B}^s_{p,1}(\R^d)$ and $ \nabla f \in \dot{B}^{\frac{d}{p}}_{p,1}(\R^d)$, then we have
\begin{align*}
& \big\| [\dot{\Delta}_k , f \cdot \nabla]g\big\|_{L^p} \le C d_k 2^{-ks} \| \nabla f\|_{\dot{B}^{\frac{d}{p}}_{p,1}} \| g\|_{\dot{B}^s_{p,1}}.
\end{align*}
Where $ (d_k)_{k \in \Z}$ denotes a sequence such that $ \| (d_k)\|_{\ell^1} \le 1$.
\end{lemma}

\vskip .1in

\section{The proof of Theorem \ref{theorem1.1}}
We are going to  prove Theorem \ref{theorem1.1} by using energy argument. For clarity, the proof is divided into  two main steps. The first step is to establish the \textit{a priori} estimates of low and high frequency, and the second step is to prove Theorem \ref{theorem1.1} by using the continuous argument. To begin with, we give some notations. Let $\Lambda\stackrel{\mathrm{def}}{=}\sqrt{-\Delta}$ and
\begin{align*}
  \delta \stackrel{\mathrm{def}}{=} \Lambda^{-1} \div \mathbb{Q}\tau,\quad \mathbb{P} \bar{\tau}\stackrel{\mathrm{def}}{=}\mathbb{P} \tau-\mathbb{P} \tau_F,
\end{align*}
where  $\mathbb{P} \tau_F$ is the  global solution of the following equation $$\partial_t \mathbb{P}\tau_F  -  \Delta \mathbb{P} \tau_F + \p \tau_F   = 0 $$
with initial data $\mathbb{P}\tau_0$.

We also choose the weighted function
\begin{align*}
f(t) \stackrel{\mathrm{def}}{=} \| \mathbb{P} \tau_F\|_{\dot{B}^{\frac{d}{p}+2}_{p,1}}  + \| \mathbb{P} \tau_F\|_{\dot{B}^{\frac{d}{p}}_{p,1}}
\end{align*}
and, for some $\lambda > 0$,  define
\begin{align}
	\u_{\lambda} \stackrel{\mathrm{def}}{=} \u \cdot \exp\bigg\{- \lambda \int_{0}^{t} f(t')dt'\bigg\}.
\end{align}
Similar notations for $ \bar{\tau}_{\lambda}, \delta_{\lambda},  \G_{\lambda}$.
To simply the notations in the energy argument, we also define the following equality.
\begin{align*}
 \mathcal E_\infty(t)\stackrel{\mathrm{def}}{=}&\| \u \|^\ell_{\widetilde{L}^{\infty}_t (\dot B^{\frac{d}{p} -1}_{p,1})} + \| \u \|^h_{\widetilde{L}^{\infty}_t (\dot B^{\frac{d}{p} + 1 }_{p,1})},\nonumber\\
\mathcal E_1^\lambda(t)\stackrel{\mathrm{def}}{=}&\| \u_{\lambda} \|^\ell_{{L}^{1}_t ({\dot B^{\frac{d}{p} + 1}_{p,1}})} + \| \u_{\lambda} \|^h_{{L}^{1}_t ({\dot B^{\frac{d}{p} + 1}_{p,1}})} + \| \mathbb{Q} \tau_{\lambda} \|^\ell_{{L}^{1}_t(\dot B^{\frac{d}{p} }_{p,1})}  \nn\\
&+ \| \mathbb{Q} \tau_{\lambda} \|^h_{{L}^{1}_t(\dot B^{\frac{d}{p} +2 }_{p,1})}  + \| \mathbb{P} \bar{\tau}_{\lambda} \|_{{L}^{1}_t(\dot B^{\frac{d}{p}  }_{p,1})} + \| \mathbb{P} \bar{\tau}_{\lambda} \|_{{L}^{1}_t(\dot B^{\frac{d}{p} +2 }_{p,1})}.
\end{align*}

In order to make full use of the structure of the system, we first applying the operators $\p, \q$ to
the first two equations of \eqref{modle} and then rewrite it into the following form
\begin{eqnarray}\label{reformulate}
\left\{\begin{aligned}
&\partial_t \u    - \mathbb{P} \div \tau =  \mathbf{N}^1,\\
& \partial_t \mathbb{Q}\tau  -  \Delta \mathbb{Q} \tau + \q \tau - \frac{\nabla \u}{2}  =  \mathbf{N}^2,\\
& \partial_t \mathbb{P}\tau  -  \Delta \mathbb{P} \tau + \p \tau  =  \mathbf{N}^3,
\end{aligned}\right.
\end{eqnarray}
where
\begin{align*}
&\mathbf{N}^1 = -\mathbb{P}(\u\cdot\nabla \u);\nn\\
&\mathbf{N}^2 = - \mathbb{Q}(\u\cdot\nabla\tau) - \mathbb{Q}(g_{\alpha}(\tau, \nabla \u));\nn\\
&\mathbf{N}^3 = \frac{(\nabla \u)^{\top}}{2}  - \mathbb{P}(\u\cdot\nabla\tau) - \mathbb{P}(g_{\alpha}(\tau, \nabla \u)).
\end{align*}

\subsection{Estimates for the incompressible part $ \mathbb{P}\tau_F$}\label{sub_pt}
In the first subsection, we shall solve the global wellposedness of the  following equation
\begin{eqnarray}\label{ptf}
\left\{\begin{aligned}
& \partial_t \mathbb{P}\tau_F  -  \Delta \mathbb{P} \tau_F + \p \tau_F   =  0\\
& \mathbb{P}\tau_F |_{t = 0} = \mathbb{P}\tau_0.
\end{aligned}\right.
\end{eqnarray}
The main result of this section is as follows.
\begin{lemma}\label{heat_kernel}
	Let $  s \in \R, T >0, 1 \le \sigma, p \le \infty$, and $ 1 \le q,r \le \infty$. $ \p \tau_F$ satisfies the equation \eqref{ptf}, then there holds
	\begin{align}\label{ptf_estimate}
	\| \p \tau_F\|_{\widetilde{\title{L}}^{q}_{T}(\dot{B}^{s + \frac{2}{q}}_{p, \sigma})} + \| \p \tau_F\|_{\widetilde{\title{L}}^{r}_{T}(\dot{B}^{s}_{p, \sigma})}  \le C \| \p \tau_0\|_{\dot{B}^{s}_{p, \sigma}}.
	\end{align}
	\begin{proof}
		Applying $ \ddk$ to \eqref{ptf} and multiplying by $ p | \dot{\Delta}_k \mathbb{P} \tau_F |^{p-2} \cdot \dot{\Delta}_k \mathbb{P} \tau_F$, then by Lemma \ref{like-bernstein}, we can get
		\begin{align*}
			\frac{d}{dt} \| \ddk \p \tau_F\|^p_{L^p} + c_p 2^{2k} \| \ddk \p \tau_F\|^p_{L^p} + p \| \ddk \p \tau_F\|^p_{L^p} \le 0.
		\end{align*}
		By simultaneously multiplying both sides of the equation above by $ \frac{1}{p \| \ddk \p \tau_F\|^{p-1}_{L^p}}$, we obtain
		\begin{align*}
			\frac{d}{dt} \| \ddk \p \tau_F\|_{L^p} + c_p 2^{2k} \| \ddk \p \tau_F\|_{L^p} +  \| \ddk \p \tau_F\|_{L^p} \le 0.
		\end{align*}
		By multiplying both sides of the equation above by $ e^{2^{2k}c_p t'}$, and then integrating over the interval from $0$ to $t$, we can obtain
		\begin{align}\label{ptf_1}
			\| \ddk \p \tau_F\|_{L^p} \lesssim e^{-2^{2k}c_p t} \| \ddk \p \tau_0\|_{L^p}.
		\end{align}
		Then, taking the $L^q$ norm of \eqref{ptf_1} with respect to $t$, we multiply both sides by $2^{ks + \frac{2k}{q}}$, there holds
		\begin{align*}
			2^{ks + \frac{2k}{q}} \| \ddk \p \tau_F\|_{L^q_T(L^p_x)} \lesssim 2^{ks} \| \ddk \p \tau_0\|_{L^p}.
		\end{align*}
		Next, taking the $ \ell^{\sigma}$ norm of the above equation with respect to $k$, we have
		\begin{align}\label{ptf2}
			\| \p \tau_F\|_{\widetilde{\title{L}}^{q}_{T}(\dot{B}^{s + \frac{2}{q}}_{p, \sigma})}  \lesssim \| \p \tau_0\|_{\dot{B}^{s}_{p, \sigma}}.
		\end{align}
		On the other hand, since $ 	\frac{d}{dt} \| \ddk \p \tau_F\|_{L^p} +  \| \ddk \p \tau_F\|_{L^p} \le 0$,  by a similar process based on equation \eqref{ptf2}, one can get that
		\begin{align}\label{ptf3}
			\| \p \tau_F\|_{\widetilde{\title{L}}^{r}_{T}(\dot{B}^{s}_{p, \sigma})}  \lesssim \| \p \tau_0\|_{\dot{B}^{s}_{p, \sigma}}.
		\end{align}
		The combination of \eqref{ptf2} and \eqref{ptf3} gives rise to \eqref{ptf_estimate}. Consequently, we complete the proof of the lemma.
	\end{proof}
\end{lemma}
\subsection{Estimates for the incompressible part $ \mathbb{P} \bar{\tau}$}\label{sub_pt1}
On the basis of superposition principle, $ \mathbb{P} \tau = \mathbb{P}\tau_F + \mathbb{P} \bar{\tau}$, one can deduce that $\mathbb{P} \bar{\tau}$  satisfies the following equation
 \begin{eqnarray}\label{ptbar}
\left\{\begin{aligned}
& \partial_t \mathbb{P} \bar{\tau}  -  \Delta \mathbb{P} \bar{\tau} + \p \bar{\tau}  =  \mathbf{N}^3,\\
& \mathbb{P} \bar{\tau} |_{t = 0} = 0.
\end{aligned}\right.
\end{eqnarray}
 In this subsection, we are concerned with the
estimates of  $ \mathbb{P} \bar{\tau}$. More precisely, we have the following lemma.
\begin{lemma}\label{le_pt}
	Under the condition in Theorem \ref{theorem1.1}, there exists a constant C such that
	\begin{align}\label{pt}
	&\| \mathbb{P} \bar{\tau}_{\lambda} \|_{\widetilde{L}^{\infty}_{t}(\dot B^{\frac{d}{p}}_{p,1})} + \lambda \| \mathbb{P} \bar{\tau}_{\lambda} \|_{{L}^{1}_{t, f}(\dot B^{\frac{d}{p}}_{p,1})} +  \| \mathbb{P} \bar{\tau}_{\lambda} \|_{{L}^{1}_{t}(\dot B^{\frac{d}{p} }_{p,1} )} + \| \mathbb{P} \bar{\tau}_{\lambda} \|_{{L}^{1}_{t}(\dot B^{\frac{d}{p}+2}_{p,1} )}
	   \nn\\
	 & \quad  \le C\Big(\| \u_{\lambda} \|^\ell_{{L}^{1}_{t}(\dot B^{\frac{d}{p}+1}_{p, 1})} + \| \u_{\lambda} \|^h_{{L}^{1}_{t}(\dot B^{\frac{d}{p}+1}_{p,1})}+ \| \u_{\lambda} \|^\ell_{{L}^{1}_{t, f}(\dot B^{\frac{d}{p} - 1}_{p,1})} +   \| \u_{\lambda} \|^h_{{L}^{1}_{t, f}(\dot B^{\frac{d}{p} + 1}_{p,1})}\Big)  + C\ea\eb.
	\end{align}
\end{lemma}
	\begin{proof}
		Applying $ \dot{\Delta}_k$ to \eqref{ptbar} and then multiplying the resulting equation by $ \exp\big\{- \lambda \int_{0}^{t} f(t')dt'\big\}$ imply that
		\begin{align}\label{ptbark}
		\partial_t \dot{\Delta}_k \mathbb{P} \bar{\tau}_{\lambda} + \lambda f(t) \dot{\Delta}_k \mathbb{P} \bar{\tau}_{\lambda} - \Delta \dot{\Delta}_k \mathbb{P} \bar{\tau}_{\lambda} + \ddk \p \bar{\tau} = \dot{\Delta}_k \mathbf{N}^3_{\lambda}.
		\end{align}
		Multiplying \eqref{ptbark} by $ p | \dot{\Delta}_k \mathbb{P} \bar{\tau}_{\lambda} |^{p-2} \cdot \dot{\Delta}_k \mathbb{P} \bar{\tau}_{\lambda}$ and integrating over $ \R^d$, on the basis of Lemma \ref{like-bernstein}, we have
		\begin{align}\label{ptk}
		&\frac{d}{dt} \| \dot{\Delta}_k \mathbb{P} \bar{\tau}_{\lambda}\|^p_{L^p} + p\lambda f(t) \| \dot{\Delta}_k \mathbb{P} \bar{\tau}_{\lambda}\|^p_{L^p} + 2^{2k} p \| \dot{\Delta}_k \mathbb{P} \bar{\tau}_{\lambda}\|^p_{L^p}  + \| \ddk \p\bar{\tau}\|^p_{L^p}\nn\\
 &\quad\lesssim \langle \dot{\Delta}_k \mathbf{N}^3_{\lambda}, p | \dot{\Delta}_k \mathbb{P} \bar{\tau}_{\lambda} |^{p-2} \cdot \dot{\Delta}_k \mathbb{P} \bar{\tau}_{\lambda}\rangle.
		\end{align}
		Thanks to the H\"oder inequality, there holds
		\begin{align}\label{ptbar_k}
		\frac{d}{dt} \| \dot{\Delta}_k \mathbb{P} \bar{\tau}_{\lambda}\|_{L^p} + \lambda f(t)\| \dot{\Delta}_k \mathbb{P} \bar{\tau}_{\lambda}\|_{L^p} + 2^{2k}  \| \dot{\Delta}_k \mathbb{P} \bar{\tau}_{\lambda}\|_{L^p} + \| \ddk \p\bar{\tau}\|_{L^p}\lesssim  \|\dot{\Delta}_k \mathbf{N}^3_{\lambda}\|_{L^p}.
		\end{align}
		Summing by frequency division and then integrating from $ 0$ to $ t$, we have
		\begin{align}\label{pt_lambda}
		\|\mathbb{P} \bar{\tau}_{\lambda}\|_{\widetilde{L}^{\infty}_{t}(\dot B^{\frac{d}{p}}_{p,1})} + \lambda \| \mathbb{P} \bar{\tau}_{\lambda} \|_{{L}^{1}_{t, f}(\dot B^{\frac{d}{p}}_{p,1})} +  \|\mathbb{P} \bar{\tau}_{\lambda}\|_{{L}^{1}_{t}(\dot B^{\frac{d}{p} }_{p,1} )}  + \|\mathbb{P} \bar{\tau}_{\lambda}\|_{{L}^{1}_{t}(\dot B^{\frac{d}{p} + 2}_{p,1} )} \lesssim \| \mathbf{N}^3_{\lambda} \|_{{L}^{1}_{t}(\dot B^{\frac{d}{p}}_{p,1})}.
		\end{align}
		Next, we bound terms in $ \mathbf{N}^3_{\lambda}$, For the first term of $ \mathbf{N}^3_{\lambda}$, applying Lemma \ref{embedding} gives
		\begin{align}\label{h31_lambda}
		\| \u_{\lambda} \|_{{L}^{1}_{t}(\dot B^{\frac{d}{p}+1}_{p,1})} &\le C\big(\| \u_{\lambda} \|^\ell_{{L}^{1}_{t}(\dot B^{\frac{d}{p}+1}_{p,1})} + \| \u_{\lambda} \|^h_{{L}^{1}_{t}(\dot B^{\frac{d}{p}+1}_{p,1})}\big).
		\end{align}
		For the second term and third term of $ \mathbf{N}^3_{\lambda}$, by using Lemmas \ref{le2.6} and  \ref{embedding}, we have
		\begin{align*}
			\| \u \cdot \nabla \tau\|_{\dot B^{\frac{d}{p}}_{p,1}} &\lesssim \| \u \|_{\dot B^{\frac{d}{p} }_{p,1}} \| \tau\|_{\dot B^{\frac{d}{p}+1}_{p,1}}\nn\\
			& \lesssim (\| \u \|^\ell_{\dot B^{\frac{d}{p}-1 }_{p,1}} + \| \u \|^h_{\dot B^{\frac{d}{p}  +1}_{p,1}}) (\| \mathbb{Q} \tau \|^\ell_{\dot B^{\frac{d}{p}  }_{p,1}} + \| \mathbb{Q} \tau \|^h_{\dot B^{\frac{d}{p} +2 }_{p,1}} + \| \mathbb{P} \tau_F \|_{\dot B^{\frac{d}{p} +1 }_{p,1}} + \| \mathbb{P} \bar{\tau} \|_{\dot B^{\frac{d}{p}+1}_{p,1}})\nn\\
			&\lesssim (\| \u \|^\ell_{\dot B^{\frac{d}{p}-1 }_{p,1}} + \| \u \|^h_{\dot B^{\frac{d}{p}  +1}_{p,1}})( \| (\mathbb{P} \tau_F \|_{\dot B^{\frac{d}{p} }_{p,1}} + \| \mathbb{P} \tau_F \|_{\dot B^{\frac{d}{p}+2}_{p,1}})\nn\\
			&\quad+(\| \u \|^\ell_{\dot B^{\frac{d}{p}-1 }_{p,1}} + \| \u \|^h_{\dot B^{\frac{d}{p}  +1}_{p,1}})(\| \mathbb{Q} \tau \|^\ell_{\dot B^{\frac{d}{p}  }_{p,1}} + \| \mathbb{Q} \tau \|^h_{\dot B^{\frac{d}{p} +2 }_{p,1}} + \|  \p \bar{\tau} \|_{\dot B^{\frac{d}{p} }_{p,1}} + \| \mathbb{P} \bar{\tau} \|_{\dot B^{\frac{d}{p}+2}_{p,1}}),
		\end{align*}
		and
		\begin{align*}
			\| g_{\alpha} (\tau, \nabla \u)\|_{\dot B^{\frac{d}{p}}_{p,1}} & \lesssim \| \u \|_{\dot B^{\frac{d}{p}+1 }_{p,1}} \| \tau\|_{\dot B^{\frac{d}{p}}_{p,1}}\nn\\
			& \lesssim (\| \u \|^\ell_{\dot B^{\frac{d}{p} -1}_{p,1}} + \| \u \|^h_{\dot B^{\frac{d}{p} +1 }_{p,1}}) (\| \mathbb{Q} \tau \|^\ell_{\dot B^{\frac{d}{p}  }_{p,1}} + \| \mathbb{Q} \tau \|^h_{\dot B^{\frac{d}{p} +2 }_{p,1}} + \| \mathbb{P} \tau_F \|_{\dot B^{\frac{d}{p}  }_{p,1}} + \| \mathbb{P} \bar{\tau} \|_{\dot B^{\frac{d}{p}}_{p,1}}).
		\end{align*}

Hence, according to the  definition of the weighted Chemin-Lerner norm, we further get
\begin{align}\label{jiaquanguji}
&\int_{0}^{t} \| (\u \cdot \nabla \tau)_{\lambda}\|_{\dot B^{\frac{d}{p}}_{p,1}} \,dt'
+\int_{0}^{t} \| (g_{\alpha} (\tau, \nabla \u))_{\lambda}\|_{\dot B^{\frac{d}{p}}_{p,1}} \,dt'
\nn\\
&\quad\lesssim\int_{0}^{t} (\| \u_{\lambda} \|^\ell_{\dot B^{\frac{d}{p}-1 }_{p,1}} + \| \u_{\lambda} \|^h_{\dot B^{\frac{d}{p}  +1}_{p,1}})( \| (\mathbb{P} \tau_F \|_{\dot B^{\frac{d}{p} }_{p,1}} + \| \mathbb{P} \tau_F \|_{\dot B^{\frac{d}{p}+2}_{p,1}})\,dt' \nn\\
&\qquad+\int_{0}^{t} (\| \u \|^\ell_{\dot B^{\frac{d}{p}-1 }_{p,1}} + \| \u \|^h_{\dot B^{\frac{d}{p}  +1}_{p,1}})(\| \mathbb{Q} \tau_{\lambda} \|^\ell_{\dot B^{\frac{d}{p}  }_{p,1}} + \| \mathbb{Q} \tau_{\lambda} \|^h_{\dot B^{\frac{d}{p} +2 }_{p,1}} + \|  \p \bar{\tau}_{\lambda} \|_{\dot B^{\frac{d}{p} }_{p,1}} + \| \mathbb{P} \bar{\tau}_{\lambda} \|_{\dot B^{\frac{d}{p}+2}_{p,1}})\,dt' \nn\\
&\quad\lesssim
 \| \u_{\lambda} \|^\ell_{{L}^{1}_{t, f}(\dot B^{\frac{d}{p} - 1}_{p,1})} +   \| \u_{\lambda} \|^h_{{L}^{1}_{t, f}(\dot B^{\frac{d}{p} + 1}_{p,1})}  + \ea\eb.
\end{align}

Inserting \eqref{h31_lambda} and \eqref{jiaquanguji} into \eqref{pt_lambda}, we can derive  \eqref{pt}. Consequently, we complete the proof of the lemma.
	\end{proof}

\subsection{Estimates of $(\u_{\lambda}, \delta_{\lambda})$ in Low-frequency}\label{sublow_fre}
In this subsection, we are concerned with the
estimates of $ (\u_{\lambda}, \delta_{\lambda})$ in low frequency.
First,   setting $$ \mathbf{v} = \Lambda^{-1} \p \div \tau,$$
 and then
 applying operator $ \Lambda^{-1} \p \div$ to the second equation of \eqref{modle},  we obtain the following new system:
\begin{align}\label{linear1}
\left\{\begin{aligned}
&\partial_t \u- \Lambda \mathbf{v} = \mathbf{H},\\
& \partial_t \mathbf{v}  + \Lambda^2 \mathbf{v}  + \mathbf{v}   + \frac{1}{2} \Lambda \u = \mathbf{M},\\
& \div \u = 0,
\end{aligned}\right.
\end{align}
where $$ \mathbf{H} = -\p (\u \cdot \nabla \u) \quad \text{and} \quad \mathbf{M} = -\Lambda^{-1} \p \div (\u \cdot \nabla \tau + g_{\alpha} (\tau, \nabla \u)).$$
Next, considering the linearized equations of  \eqref{linear1}, we have
\begin{align*}
\left\{\begin{aligned}
&\partial_t \u- \Lambda \mathbf{v} = 0,\\
& \partial_t \mathbf{v}  + \Lambda^2 \mathbf{v}  + \mathbf{v}   + \frac{1}{2} \Lambda \u = 0.
\end{aligned}\right.
\end{align*}
By applying the Fourier transform to the aforementioned linearized equations, we obtain
\begin{align}\label{linear}
	\frac{d}{d t}\binom{\widehat{\u}}{\widehat{\mathbf{v}}}=\mathbb{A}(\xi)\binom{\widehat{\u}}{\widehat{\mathbf{v}}} \quad \text { with } \quad \mathbb{A}(\xi)=\left(\begin{array}{cc}
	0  \mathbb{I}_d  &  |\xi|\mathbb{I}_d \\
	-\frac{1}{2}|\xi| \mathbb{I}_d & -\left(|\xi|^2+1\right) \mathbb{I}_d
	\end{array}\right) .
\end{align}
Moreover, a simple computation implies that
system \eqref{linear} has the following two eigenvalues $$
	\lambda_{ \pm}=\frac{-\left(|\xi|^2+1\right) \pm \sqrt{\left(|\xi|^2+1\right)^2-2 |\xi|^2}}{2} .
	$$
\begin{lemma}\label{low_fre}
	
	Let $\widehat{\mathbb{G}}_{\u, \mathbf{v}}(t, \xi)=e^{\mathbb{A}(\xi)t}$ be the Green's matrix for System \eqref{linear}. Then the solution of System \eqref{linear} with initial data $\left.(\widehat{\u}, \widehat{\mathbf{v}})\right|_{t=0}=\left(\widehat{\u}_0, \widehat{\mathbf{v}}_0\right)$ has the following exact representation:
	\begin{align}\label{linear2}
		\binom{\widehat{\u}}{\widehat{\mathbf{v}}}=\widehat{\mathbb{G}}_{\u, \mathbf{v}}(t, \xi)\binom{\widehat{\u}_0}{\widehat{\mathbf{v}}_0},
	\end{align}
	where
\begin{align}\label{G}
	\widehat{\mathbb{G}}_{\u, \mathbf{v}}(t, \xi)=\left(\begin{array}{cc}
	\mathcal{G}_3(t, \xi) \mathbb{I}_d & |\xi| \mathcal{G}_1(t, \xi) \mathbb{I}_d \\
	-\frac{1}{2}|\xi| \mathcal{G}_1(t, \xi) \mathbb{I}_d & \mathcal{G}_2(t, \xi) \mathbb{I}_d
	\end{array}\right)
\end{align}
	with
	$$
	\mathcal{G}_1(t, \xi)=\frac{e^{\lambda_{+} t}-e^{\lambda_{-} t}}{\lambda_{+}-\lambda_{-}}, \quad \mathcal{G}_2(t, \xi)=\frac{\lambda_{+} e^{\lambda_{+} t}-\lambda_{-} e^{\lambda_{-} t}}{\lambda_{+}-\lambda_{-}} , \quad\mathcal{G}_3(t, \xi)=\frac{\lambda_{+} e^{\lambda_{-} t}-\lambda_{-} e^{\lambda_{+} t}}{\lambda_{+}-\lambda_{-}}.
	$$
	
 Furthermore, let $\mathbb{C}$ be a ring centered at $0$ in $\mathbb{R}^d$, and if supp $\widehat{\phi} \subseteq$ $\theta \mathbb{C}$ for some constant $\theta>0$. For any $R>0$, there exist positive constants $C$ and $c$ depending on  $d$, such that if $\theta \leq R$, then for any $1 \leq p \leq 2d$,
 \begin{align}\label{low}
 	\left\| \check{\mathcal{G}}_i \ast \phi\right\|_{L^p} \leq C e^{-c \theta^2 t}\|\phi\|_{L^p}  \quad \text{for} \quad i = 1,2,3.
 \end{align}
 \begin{proof}
 	The proof of \eqref{linear2} is straightforward, while the detailed proof of \eqref{low} can be found in \cite{zi2024JDE}.
 \end{proof}
\end{lemma}

Next, we use Lemma \ref{low_fre} to obtain the dissipation estimate for the low-frequency component of $(\u_{\lambda}, \delta_{\lambda})$. More precisely, we have the following lemma.
\begin{lemma}\label{le_ut_low}
Under the condition in Theorem \ref{theorem1.1}, there exists a constant $ C$ such that
\begin{align}\label{ut_low}
& \| \u_{\lambda}\|^\ell_{\widetilde{L}^{\infty}_{t}(\dot B^{\frac{d}{p} - 1}_{p,1})} + \|  \delta_{\lambda}\|^\ell_{\widetilde{L}^{\infty}_{t}(\dot B^{\frac{d}{p}}_{p,1})} + \lambda \| \u_{\lambda} \|^\ell_{{L}^{1}_{t, f}(\dot B^{\frac{d}{p} - 1}_{p,1})} + \lambda \| \delta_{\lambda} \|^\ell_{{L}^{1}_{t, f}(\dot B^{\frac{d}{p} }_{p,1})} +  \| \u_{\lambda}\|^\ell_{{L}^{1}_{t}(\dot B^{\frac{d}{p} + 1}_{p,1} )}+  \| \delta_{\lambda}\|^\ell_{{L}^{1}_{t}(\dot B^{\frac{d}{p} }_{p,1} )}  \nn\\
 &\quad\le C \Big(\| \u_0 \|^\ell_{\dot{B}^{\frac{d}{p}-1}_{p,1}} + \|  \delta_0 \|^\ell_{\dot{B}^{\frac{d}{p}}_{p,1}} +  \| \u_{\lambda} \|^\ell_{{L}^{1}_{t, f}(\dot B^{\frac{d}{p} - 1}_{p,1})} +  \| \u_{\lambda} \|^h_{{L}^{1}_{t, f}(\dot B^{\frac{d}{p} + 1}_{p,1})}\Big) +C\ea\eb.
\end{align}
\begin{proof}
	First, we localize System \eqref{linear1} as follow:
	\begin{align}\label{linear3}
	\left\{\begin{aligned}
	&\partial_t \dot{\Delta}_k \u- \Lambda \dot{\Delta}_k \mathbf{v} = \dot{\Delta}_k \mathbf{H},\\
	& \partial_t \dot{\Delta}_k \mathbf{v}  + \Lambda^2 \dot{\Delta}_k \mathbf{v}  + \dot{\Delta}_k \mathbf{v}   + \frac{1}{2} \Lambda \dot{\Delta}_k \u = \dot{\Delta}_k \mathbf{M},\\
	& (\dot{\Delta}_k \u, \dot{\Delta}_k \tau_0)|_{t = 0} = (\dot{\Delta}_k \u_0, \dot{\Delta}_k \tau_0).
	\end{aligned}\right.
	\end{align}
	Based on \eqref{linear2} in Lemma \ref{low_fre}, we can derive the solution expression for the linearized equations. Applying the Duhamel principle, we obtain:
	\begin{align}\label{low1}
			\dot{\Delta}_k \u= & \check{\mathcal{G}}_3 \ast \dot{\Delta}_k \u_0+  \Lambda \check{\mathcal{G}_1} \ast  \dot{\Delta}_k \mathbf{v}_0 \nn\\
		& +\int_0^t \check{\mathcal{G}_3}\left(t-t^{\prime}, \cdot \right) \ast \dot{\Delta}_k \mathbf{H} \, d t^{\prime}+  \int_0^t \Lambda \check{\mathcal{G}_1}\left(t-t^{\prime}, \cdot\right) \ast \dot{\Delta}_k \mathbf{M} \, d t^{\prime}.
	\end{align}
	Next, multiplying the above equation by $ e^{-\lambda \int_{0}^{t} f(t') \, d t'}$ yields that:
\begin{align}\label{equ_l}
	 \dot{\Delta}_k \u_{\lambda} = & e^{-\lambda \int_{0}^{t} f(t') \, d t'} \check{\mathcal{G}}_3 \ast \dot{\Delta}_k \u_0 +  e^{-\lambda \int_{0}^{t} f(t') \, d t'} \Lambda \check{\mathcal{G}_1} \ast  \dot{\Delta}_k \mathbf{v}_0 \nn\\
	 & + e^{-\lambda \int_{0}^{t} f(t') \, d t'} \Big[ \int_0^t \check{\mathcal{G}_3}\left(t-s, \cdot \right) \ast \dot{\Delta}_k \mathbf{H} \, d s+   \int_0^t \Lambda \check{\mathcal{G}_1}\left(t-s, \cdot\right) \ast \dot{\Delta}_k \mathbf{M} \, d s \Big].
\end{align}
Subsequently, we define  $U_k^\lambda (t)$ as follows:
\begin{align}\label{U}
	\mathbf{U}_k^\lambda (t) = \left \| \dot{\Delta}_k \u_{\lambda} \right\|_{L_t^{\infty}\left(L^p\right)} + \left \| \lambda f(t) \|\dot{\Delta}_k \u_{\lambda}\|_{L^p} \right\|_{L_t^{1}} + 2^{2 k}\left\|\dot{\Delta}_k \u_{\lambda}\right\|_{L_l^1\left(L^p\right)}.
\end{align}

For the first term in \eqref{U}, taking the $L^p$ norm with respect to the spatial variable $x$ and the supremum with respect to the time variable $t$, and then using  \eqref{low} of Lemma \ref{low_fre}, we obtain:
\begin{align}\label{linear4}
	 \| \dot{\Delta}_k \u_{\lambda} \|_{L_t^{\infty}\left(L^p\right)} & \lesssim   \| \dot{\Delta}_k \u_{0} \|_{L^p} + 2^k \| \dot{\Delta}_k \mathbf{v}_{0} \|_{L^p}\nn\\
	 & \quad + \int_0^t e^{-c\cdot 2^{2k}(t-s)} e^{-\lambda \int_{0}^{t} f(t') \, d t'} \cdot e^{\lambda \int_{0}^{s} f(s) \, d s} \cdot \|\dot{\Delta}_k \mathbf{H}_{\lambda}\|_{L^p} \, d s\nn\\
	 & \quad + \int_0^t 2^k e^{-c\cdot 2^{2k}(t-s)} e^{-\lambda \int_{0}^{t} f(t') \, d t'} \cdot e^{\lambda \int_{0}^{s} f(s) \, d s} \cdot \|\dot{\Delta}_k \mathbf{M}_{\lambda}\|_{L^p} \, d s \nn\\
	 & \lesssim \| \dot{\Delta}_k \u_{0} \|_{L^p} + 2^k \| \dot{\Delta}_k \mathbf{v}_{0} \|_{L^p} + \| \dot{\Delta}_k \mathbf{H}_{\lambda} \|_{L_t^{1}\left(L^p\right)} + 2^k \| \dot{\Delta}_k \mathbf{M}_{\lambda} \|_{L_t^{1}\left(L^p\right)}.
\end{align}
For the second term in  \eqref{U}, similarly, one can arrive at
\begin{align}\label{linear5}
	\left \| \lambda f(t) \|\dot{\Delta}_k \u_{\lambda}\|_{L^p} \right\|_{L_t^{1}} &\lesssim  \int_0^t \lambda f(s) e^{-\lambda \int_{0}^{s} f(t') \, d t'} e^{-c\cdot 2^{2k}s} \| \dot{\Delta}_k \u_{0} \|_{L^p} \, ds \nn\\
	& \quad + \int_0^t \lambda f(s) e^{-\lambda \int_{0}^{s} f(t') \, d t'} \cdot 2^k \cdot e^{-c\cdot 2^{2k}s} \| \dot{\Delta}_k \mathbf{v}_{0} \|_{L^p} \, ds \nn\\
	& \quad + \int_0^t \lambda f(s) e^{-\lambda \int_{0}^{s} f(t') \, d t'} \int_{0}^{s} e^{-c\cdot 2^{2k}(s-s')} \| \dot{\Delta}_k \mathbf{H}(s') \|_{L^p} \, ds' \, ds \nn\\
	& \quad  + \int_0^t \lambda f(s) e^{-\lambda \int_{0}^{s} f(t') \, d t'} \int_{0}^{s} 2^k \cdot e^{-c\cdot 2^{2k}(s-s')} \| \dot{\Delta}_k \mathbf{M}(s') \|_{L^p} \, ds' \, ds \nn\\
	& \stackrel{\mathrm{def}}{=} \sum_{j =1}^{4} I_j.
\end{align}
For the terms $I_1$ and $I_2$, we can directly obtain:
\begin{align}\label{I12}
	I_1 + I_2  &\lesssim \int_0^t \lambda f(s) e^{-\lambda \int_{0}^{s} f(t') \, d t'}   \, ds \Big(\| \dot{\Delta}_k \u_{0} \|_{L^p} + 2^k \| \dot{\Delta}_k \mathbf{v}_{0} \|_{L^p}\Big)\nn\\
	& \lesssim \| \dot{\Delta}_k \u_{0} \|_{L^p} + 2^k \| \dot{\Delta}_k \mathbf{v}_{0} \|_{L^p}
\end{align}
where we have used the fact $$ \int_0^t \lambda f(s) e^{-\lambda \int_{0}^{s} f(t') \, d t'}   \, ds \le 2.$$
For the third term $I_3$, by exchanging the order of integration, one can arrive at
\begin{align*}
	I_3 &\lesssim \int_0^t \int_{0}^{s} \lambda f(s) e^{-\lambda \int_{0}^{s} f(t') \, d t'}  \| \dot{\Delta}_k \mathbf{H}(s') \|_{L^p} \, ds' \, ds\nn\\
	& = \int_0^t \int_{0}^{s} \lambda f(s) e^{-\lambda \int_{s'}^{s} f(t') \, d t'}  \| \dot{\Delta}_k \mathbf{H}_{\lambda} \|_{L^p} \, ds' \, ds\nn\\
	& = \int_0^t \int_{s'}^{t} \lambda f(s) e^{-\lambda \int_{s'}^{s} f(t') \, d t'}  \| \dot{\Delta}_k \mathbf{H}_{\lambda} \|_{L^p}  \, ds \, ds'\nn\\
	& \lesssim \| \dot{\Delta}_k \mathbf{H}_{\lambda} \|_{L^1_t \left( L^p\right)}.
\end{align*}
Similarly, we can also obtain the estimate of $I_4$ as follows
\begin{align}\label{I4}
	I_4  \lesssim 2^k \| \dot{\Delta}_k \mathbf{M}_{\lambda} \|_{L^1_t \left( L^p\right)}.
\end{align}
Substituting \eqref{I12}--\eqref{I4} into \eqref{linear5} gives rise to
\begin{align}\label{linear6}
	\left \| \lambda f(t) \|\dot{\Delta}_k \u_{\lambda}\|_{L^p} \right\|_{L_t^{1}} \lesssim \| \dot{\Delta}_k \u_{0} \|_{L^p} + 2^k \| \dot{\Delta}_k \mathbf{v}_{0} \|_{L^p} + \| \dot{\Delta}_k \mathbf{H}_{\lambda} \|_{L^1_t \left( L^p\right)} + 2^k \| \dot{\Delta}_k \mathbf{M}_{\lambda} \|_{L^1_t \left( L^p\right)}.
\end{align}
For the last term in  \eqref{U}, similarly, one can arrive at
\begin{align}\label{linear7}
	2^{2 k}\left\|\dot{\Delta}_k \u_{\lambda}\right\|_{L_l^1\left(L^p\right)}  &\lesssim  \int_0^t 2^{2k} e^{-c\cdot 2^{2k}s} \| \dot{\Delta}_k \u_{0} \|_{L^p} \, ds \nn\\
	& \quad + \int_0^t   2^{3k}  \cdot e^{-c\cdot 2^{2k}s} \| \dot{\Delta}_k \mathbf{v}_{0} \|_{L^p} \, ds \nn\\
	& \quad + 2^{2k} \int_0^t  e^{-\lambda \int_{0}^{s} f(t') \, d t'} \int_{0}^{s} e^{-c\cdot 2^{2k}(s-s')} \| \dot{\Delta}_k \mathbf{H}(s') \|_{L^p} \, ds' \, ds \nn\\
	& \quad  + 2^{3k} \int_0^t  e^{-\lambda \int_{0}^{s} f(t') \, d t'} \int_{0}^{s}  \cdot e^{-c\cdot 2^{2k}(s-s')} \| \dot{\Delta}_k \mathbf{M}(s') \|_{L^p} \, ds' \, ds \nn\\
	& \lesssim \| \dot{\Delta}_k \u_{0} \|_{L^p} + 2^k \| \dot{\Delta}_k \mathbf{v}_{0} \|_{L^p} + \| \dot{\Delta}_k \mathbf{H}_{\lambda} \|_{L^1_t \left( L^p\right)} + 2^k \| \dot{\Delta}_k \mathbf{M}_{\lambda} \|_{L^1_t \left( L^p\right)}
\end{align}
where we have used the fact $$ \int_{0}^{t} 2^{2k} e^{-c\cdot 2^{2k}s} \, ds \le C.$$
Inserting \eqref{linear4}, \eqref{linear6} and \eqref{linear7} into \eqref{U}, we obtain:
\begin{align}\label{low2}
	\mathbf{U}_k^\lambda (t) \lesssim \| \dot{\Delta}_k \u_{0} \|_{L^p} + 2^k \| \dot{\Delta}_k \mathbf{v}_{0} \|_{L^p} + \| \dot{\Delta}_k \mathbf{H}_{\lambda} \|_{L^1_t \left( L^p\right)} + 2^k \| \dot{\Delta}_k \mathbf{M}_{\lambda} \|_{L^1_t \left( L^p\right)}.
\end{align}
Multiplying the above equation by $ 2^{(\frac{d}{p}-1)k}$ and then summing over $k$ for $k \le k_0$ yields that
\begin{align}\label{low3}
		&\|\u_{\lambda}\|^\ell_{\widetilde{L}^\infty_t(\dot{B}^{\frac{d}{p}-1}_{p,1})} + \lambda \| \u_{\lambda} \|^\ell_{{L}^{1}_{t, f}(\dot B^{\frac{d}{p} - 1}_{p,1})} + \int_{0}^{t} \| \u_{\lambda}\|^\ell_{\dot B^{\frac{d}{p} + 1}_{p,1}} \,dt' \nn\\
	 & \quad \lesssim  \| \u_0 \|^\ell_{\dot{B}^{\frac{d}{p}-1}_{p,1}} + \| \mathbf{v}_0 \|^\ell_{\dot{B}^{\frac{d}{p}}_{p,1}} + \int_{0}^{t} \| \mathbf{H}_{\lambda}\|^\ell_{\dot B^{\frac{d}{p} - 1}_{p,1}} \,dt'  + \int_{0}^{t} \| \mathbf{M}_{\lambda}\|^\ell_{\dot B^{\frac{d}{p}}_{p,1}} \,dt'.
\end{align}

On the other hand, we consider the compressible part of $\tau$, applying operator $\la^{-1} \div$ to the second equation of \eqref{reformulate} and using the definition $\delta = \la^{-1} \div \q \tau$ give rise to
\begin{align}\label{eq_delta}
	\partial_t \delta - \Delta  \delta +\delta + \frac{1}{2} \Lambda \u  =   \Lambda^{-1} \div \mathbf{N}^2.
\end{align}
Applying $\ddk\exp\{- \lambda \int_{0}^{t} f(t')dt'\}$ to the above equation
and multiplying  \eqref{du_k} by $  | \dot{\Delta}_k \delta_{\lambda} |^{p-2} \cdot \dot{\Delta}_k \delta_{\lambda}$  and integration by parts over $ \R^d$, one can get
\begin{align}\label{deltak_l}
		&\frac{1}{p} \frac{d}{dt} \| \dot{\Delta}_k \delta_{\lambda}\|^p_{L^p} + \lambda f(t) \| \dot{\Delta}_k \delta_{\lambda}\|^p_{L^p} + \| \nabla \dot{\Delta}_k \delta_{\lambda} \|^p_{L^p} + \| \ddk \delta_{\lambda}\|^p_{L^p} \nn\\
		&\quad = - \frac{1}{2} \int_{\R^d}   \dot{\Delta}_k \Lambda \u_{\lambda} \cdot |\dot{\Delta}_k \delta_{\lambda}|^{p-2} \cdot \dot{\Delta}_k \delta_{\lambda} \,dx  + \int_{\R^d}   \Lambda^{-1} \div \dot{\Delta}_k \mathbf{N}^2_{\lambda} \cdot |\dot{\Delta}_k \delta_{\lambda}|^{p-2} \cdot \dot{\Delta}_k \delta_{\lambda} \,dx.
\end{align}
Using H\"older's inequality, we can obtain
\begin{align}\label{low4}
	\frac{d}{dt} \| \dot{\Delta}_k \delta_{\lambda}\|_{L^p} + \lambda f(t) \| \dot{\Delta}_k \delta_{\lambda}\|_{L^p} + 2^{2k} \| \dot{\Delta}_k \delta_{\lambda} \|_{L^p} + \| \ddk \delta_{\lambda}\|_{L^p}  \lesssim    2^k \| \dot{\Delta}_k \u_{\lambda} \|_{L^p} +   \|\dot{\Delta}_k \mathbf{N}^2_{\lambda}\|_{L^p}.
\end{align}
Multiplying the above equation by $ 2^{\frac{d}{p}k}$,  taking the summation of $k \le k_0$,  and then integrating from $0$ to $t$ yields
\begin{align}\label{low5}
		&\|\delta_{\lambda}\|^\ell_{\widetilde{L}^\infty_t(\dot{B}^{\frac{d}{p}}_{p,1})} + \lambda \| \delta_{\lambda} \|^\ell_{{L}^{1}_{t, f}(\dot B^{\frac{d}{p}}_{p,1})} + \int_{0}^{t} \| \delta_{\lambda}\|^\ell_{\dot B^{\frac{d}{p} + 2}_{p,1}} + \| \delta_{\lambda}\|^\ell_{\dot B^{\frac{d}{p}}_{p,1}} \,dt' \nn\\
	& \quad \lesssim  \| \delta_0 \|^\ell_{\dot{B}^{\frac{d}{p}}_{p,1}} + \int_{0}^{t} \| \u_{\lambda}\|^\ell_{\dot B^{\frac{d}{p} + 1}_{p,1}} \,dt'  + \int_{0}^{t} \| \mathbf{N}^2_{\lambda}\|^\ell_{\dot B^{\frac{d}{p}}_{p,1}} \,dt'.
\end{align}
Multiplying \eqref{low3} by an appropriate constant $C$ and adding it to \eqref{low5}, one can arrive at
\begin{align}\label{udelta_l}
	&\| \u_{\lambda}  \|^\ell_{\widetilde{L}^\infty_t(\dot{B}^{\frac{d}{p}-1}_{p,1})} + \|\delta_{\lambda}\|^\ell_{\widetilde{L}^\infty_t(\dot{B}^{\frac{d}{p}}_{p,1})} + \lambda \Big(\| \u_{\lambda}\|^\ell_{{L}^{1}_{t, f}(\dot B^{\frac{d}{p} - 1}_{p,1})} +  \| \delta_{\lambda} \|^\ell_{{L}^{1}_{t, f}(\dot B^{\frac{d}{p}}_{p,1})}\Big) + \int_{0}^{t}( \| \u_{\lambda}\|^\ell_{\dot B^{\frac{d}{p} + 1}_{p,1}} + \|  \delta_{\lambda}\|^\ell_{\dot B^{\frac{d}{p}}_{p,1}}) \,dt'\nn\\
	&\quad \lesssim  \| \u_0 \|^\ell_{\dot{B}^{\frac{d}{p}-1}_{p,1}} + \| \delta_0 \|^\ell_{\dot{B}^{\frac{d}{p}}_{p,1}} + \int_{0}^{t} \| \mathbf{N}^1_{\lambda}\|^\ell_{\dot B^{\frac{d}{p} - 1}_{p,1}} \,dt' + \int_{0}^{t} \| \mathbf{N}^2_{\lambda}\|^\ell_{\dot B^{\frac{d}{p}}_{p,1}} \,dt'
\end{align}
where we have used the fact $ \| \Lambda \mathbf{M}_{\lambda}\|^\ell_{\dot B^{\frac{d}{p}}_{p,1}} \lesssim \|  \mathbf{N}^2_{\lambda}\|^\ell_{\dot B^{\frac{d}{p} }_{p,1}}$ and $ \| \mathbf{v}_0 \|^\ell_{\dot{B}^{\frac{d}{p}}_{p,1}} \le \| \delta_0 \|^\ell_{\dot{B}^{\frac{d}{p}}_{p,1}}.$

Next we need to deal with the nonlinear terms in \eqref{udelta_l}. At first, applying Lemma \ref{le2.6} and Lemma \ref{embedding} gives rise to
\begin{align}\label{udu_l}
	\| \u \cdot \nabla \u\|^\ell_{\dot B^{\frac{d}{p} - 1}_{p,1}} &\lesssim (\|\nabla \u \|^\ell_{\dot B^{\frac{d}{p} - 1}_{p,1}} + \|\nabla \u \|^h_{\dot B^{\frac{d}{p} - 1}_{p,1}}) \| \u\|_{\dot B^{\frac{d}{p}}_{p,1}}\nn\\
	  &\lesssim   \| \u \|_{\dot B^{\frac{d}{p} }_{p,1}}^2
	 \lesssim \| \u \|^\ell_{\dot B^{\frac{d}{p} -1}_{p,1}} \| \u \|^\ell_{\dot B^{\frac{d}{p} +1}_{p,1}} + \| \u \|^h_{\dot B^{\frac{d}{p} +1 }_{p,1}} \| \u \|^h_{\dot B^{\frac{d}{p} + 1}_{p,1}}
\end{align}
which implies
\begin{align}\label{h1_l}
	\| \mathbf{N}^1_{\lambda} \|^\ell_{{L}^{1}_t ({\dot B^{\frac{d}{p} - 1}_{p,1}})} \lesssim \| \u\|^\ell_{\widetilde{L}^{\infty}_t ({\dot B^{\frac{d}{p} - 1}_{p,1}})} \| \u_{\lambda} \|^\ell_{{L}^{1}_t ({\dot B^{\frac{d}{p} + 1}_{p,1}})} + \| \u\|^h_{\widetilde{L}^{\infty}_t ({\dot B^{\frac{d}{p} + 1}_{p,1}})} \| \u_{\lambda} \|^h_{{L}^{1}_t ({\dot B^{\frac{d}{p} + 1}_{p,1}})}.
\end{align}
Noticing the fact that operator $\la^{-1}\div $ is a zero order Fourier multiplier, we infer from Lemma \ref{le2.6} again that
\begin{align*}
	&\| \Lambda^{-1} \div (\u \cdot \nabla \tau )\|^\ell_{\dot B^{\frac{d}{p} - 1}_{p,1}}
\nn\\&\quad\lesssim (\| \u \|^\ell_{\dot B^{\frac{d}{p} -1}_{p,1}} + \| \u \|^h_{\dot B^{\frac{d}{p} + 1 }_{p,1}}) \| \tau \|_{\dot B^{\frac{d}{p} +1 }_{p,1}}\nn\\
	& \quad\lesssim (\| \u \|^\ell_{\dot B^{\frac{d}{p} -1}_{p,1}} + \| \u \|^h_{\dot B^{\frac{d}{p} + 1 }_{p,1}}) (\| \mathbb{Q} \tau \|^\ell_{\dot B^{\frac{d}{p} +1 }_{p,1}} + \| \mathbb{Q} \tau \|^h_{\dot B^{\frac{d}{p} +2 }_{p,1}} + \| \mathbb{P} \tau_F \|_{\dot B^{\frac{d}{p} +1 }_{p,1}} + \| \mathbb{P} \bar{\tau} \|_{\dot B^{\frac{d}{p} +1 }_{p,1}})\nn\\
	& \quad\lesssim  (\| \u \|^\ell_{\dot B^{\frac{d}{p} -1}_{p,1}} + \| \u \|^h_{\dot B^{\frac{d}{p} + 1 }_{p,1}}) \| \mathbb{P} \tau_F \|_{\dot B^{\frac{d}{p} +1 }_{p,1}}\nn\\
	& \qquad +(\| \u \|^\ell_{\dot B^{\frac{d}{p} -1}_{p,1}} + \| \u \|^h_{\dot B^{\frac{d}{p} + 1 }_{p,1}}) (\| \mathbb{Q} \tau \|^\ell_{\dot B^{\frac{d}{p} }_{p,1}} + \| \mathbb{Q} \tau \|^h_{\dot B^{\frac{d}{p} +2 }_{p,1}}  + \| \mathbb{P} \bar{\tau} \|_{\dot B^{\frac{d}{p} +1 }_{p,1}}) \nn\\
&\quad\lesssim (\| \u \|^\ell_{\dot B^{\frac{d}{p} -1}_{p,1}} + \| \u \|^h_{\dot B^{\frac{d}{p} + 1 }_{p,1}})( \| \mathbb{P} \tau_F \|_{\dot B^{\frac{d}{p} }_{p,1}} + \| \mathbb{P} \tau_F \|_{\dot B^{\frac{d}{p}+2}_{p,1}})\nn\\
			&\qquad+(\| \u \|^\ell_{\dot B^{\frac{d}{p} -1}_{p,1}} + \| \u \|^h_{\dot B^{\frac{d}{p} + 1 }_{p,1}}) (\| \mathbb{Q} \tau \|^\ell_{\dot B^{\frac{d}{p}  }_{p,1}} + \| \mathbb{Q} \tau \|^h_{\dot B^{\frac{d}{p} +2 }_{p,1}} + \|  \p \bar{\tau} \|_{\dot B^{\frac{d}{p} }_{p,1}} + \| \mathbb{P} \bar{\tau} \|_{\dot B^{\frac{d}{p}+2}_{p,1}}).
\end{align*}

Hence, we get by a similar derivation of \eqref{jiaquanguji} that
\begin{align}\label{udtlambda_l}
	&\int_{0}^{t} \| \Lambda^{-1} \div (\u \cdot \nabla \tau )_{\lambda}\|^\ell_{\dot B^{\frac{d}{p} - 1}_{p,1}} \,dt' \lesssim \| \u_{\lambda} \|^\ell_{{L}^{1}_{t, f}(\dot B^{\frac{d}{p} - 1}_{p,1})} + \| \u_{\lambda} \|^h_{{L}^{1}_{t, f}(\dot B^{\frac{d}{p} + 1}_{p,1})}+\ea\eb.
\end{align}
For the last term in $\la^{-1} \div \mathbf{N}^2_{\lambda}$, we infer from Lemma \ref{le2.6} that
\begin{align*}
	&\| \Lambda^{-1} \div g_{\alpha}(\tau ,\nabla \u )\|^\ell_{\dot B^{\frac{d}{p} }_{p,1}} \nn\\
	&\quad \lesssim \| \Lambda^{-1} \div g_{\alpha}(\tau ,\nabla \u )\|^\ell_{\dot B^{\frac{d}{p} - 1}_{p,1}} \nn\\
	& \quad\lesssim (\| \u \|^\ell_{\dot B^{\frac{d}{p}-1 }_{p,1}} + \| \u \|^h_{\dot B^{\frac{d}{p}+1  }_{p,1}}) (\| \mathbb{Q} \tau \|^\ell_{\dot B^{\frac{d}{p}  }_{p,1}} + \| \mathbb{Q} \tau \|^h_{\dot B^{\frac{d}{p}  +2}_{p,1}} + \| \mathbb{P} \tau_F \|_{\dot B^{\frac{d}{p}  }_{p,1}} + \| \mathbb{P} \bar{\tau} \|_{\dot B^{\frac{d}{p}}_{p,1}})
\end{align*}
from which and \eqref{jiaquanguji}, one has
\begin{align}\label{tdu_ll}
	&\int_{0}^{t} \| \Lambda^{-1} \div g_{\alpha}(\tau ,\nabla \u )_{\lambda}\|^\ell_{\dot B^{\frac{d}{p} }_{p,1}} \,dt' \lesssim \| \u_{\lambda} \|^\ell_{{L}^{1}_{t, f}(\dot B^{\frac{d}{p} - 1}_{p,1})} + \| \u_{\lambda} \|^h_{{L}^{1}_{t, f}(\dot B^{\frac{d}{p} + 1}_{p,1})}+\ea\eb.
\end{align}
Combining \eqref{udtlambda_l} and \eqref{tdu_ll} together , we can get
\begin{align}\label{h2_l}
	& \int_{0}^{t} \| \Lambda^{-1} \div \mathbf{N}^2_{\lambda} \|^\ell_{\dot B^{\frac{d}{p}}_{p,1}} \,dt' \lesssim   \| \u_{\lambda} \|^\ell_{{L}^{1}_{t, f}(\dot B^{\frac{d}{p} - 1}_{p,1})}  +  \| \u_{\lambda} \|^h_{{L}^{1}_{t, f}(\dot B^{\frac{d}{p} + 1}_{p,1})} +\ea\eb.
\end{align}
Inserting \eqref{h1_l} and \eqref{h2_l} into \eqref{udelta_l}, we can get \eqref{ut_low}. Consequently, we complete the proof of the lemma.
\end{proof}
\end{lemma}

Next we shall
establish the high-frequency estimates of $( \u_{\lambda}, \delta_{\lambda}) $.
\subsection{Estimates  of $( \u_{\lambda}, \delta_{\lambda}) $ in high-frequency}\label{sub_high} In this subsection, since the equation of velocity lacks dissipation, we aim to find the damping effect of the velocity field in high frequency.
\begin{lemma}\label{le_dut_h}
	Under the condition in Theorem \ref{theorem1.1}, there exists a constant $ C$ such that
	\begin{align}\label{ut_h}
	&\| (\Lambda \u_{\lambda}, \delta_{\lambda})\|^h_{\widetilde{L}^{\infty}_{t}(\dot B^{\frac{d}{p}}_{p,1})}   +  \lambda \|(\Lambda \u_{\lambda}, \delta_{\lambda}) \|^h_{{L}^{1}_{t, f}(\dot B^{\frac{d}{p}}_{p,1})} + \| \u_{\lambda} \|^h_{{L}^{1}_{t}(\dot B^{\frac{d}{p}+1}_{p,1})} + \|  \delta_{\lambda} \|^h_{{L}^{1}_{t}(\dot B^{\frac{d}{p} + 2}_{p,1})}\nn\\
	&\quad\le  \| (\Lambda \mathbf{u}_0,\mathbb{Q}\tau_0)\|^h_{\dot B^{\frac{d}{p}}_{p,1}}  + C \| \u_{\lambda} \|^\ell_{{L}^{1}_{t, f}(\dot B^{\frac{d}{p} - 1}_{p,1})} + C \| \Lambda\u_{\lambda} \|^h_{{L}^{1}_{t, f}(\dot B^{\frac{d}{p} }_{p,1})} + C\ea\eb.
	\end{align}
\begin{proof}
	In order to find the damping effect of the velocity field in high frequency, we need to introduce the new unknown $$ \G \stackrel{\mathrm{def}}{=} -\mathbb{Q} \tau - \frac{1}{2} \nabla \Delta^{-1} \u.$$
It is direct to see that $\u$ satisfy the following damping equation
	\begin{align}
	&\partial_t \u + \frac{1}{2} \u = -\mathbb{P} (\u \cdot \nabla \u) - \mathbb{P} \div \G.
	\end{align}
	Applying the operator $ \dot{\Delta}_k \Lambda $  to the above equation and multiplying $ \exp\{- \lambda \int_{0}^{t} f(t')dt'\}$, ($ \lambda$ will be chosen later),  we have
	\begin{align}\label{du_k}
		\partial_t \dot{\Delta}_k \Lambda \u_{\lambda} + \frac{1}{2} \Lambda \u_{\lambda} = ([\u \cdot \nabla, \Lambda \dot{\Delta}_k \mathbb{P}] \u)_{\lambda} - \u \cdot \nabla \Lambda \dot{\Delta}_k \u_{\lambda}  - \dot{\Delta}_k \Lambda \mathbb{P} \div \G_{\lambda}.
	\end{align}
Multiplying  \eqref{du_k} by $  | \dot{\Delta}_k \Lambda \u_{\lambda} |^{p-2} \cdot \dot{\Delta}_k \Lambda \u_{\lambda}$  and integration by parts over $ \R^d$,  we have
	\begin{align}\label{du}
		&\frac{1}{p} \frac{d}{dt} \| \dot{\Delta}_k \Lambda \u_{\lambda}\|^p_{L^p} + \lambda f(t) \| \dot{\Delta}_k \Lambda \u_{\lambda}\|^p_{L^p} + \frac{1}{2} \| \dot{\Delta}_k \Lambda \u_{\lambda}\|^p_{L^p} \nn\\
		&\quad\le \|(([\u \cdot \nabla, \Lambda \dot{\Delta}_k \mathbb{P}] \u)_{\lambda},  \dot{\Delta}_k \Lambda \mathbb{P} \div \G_{\lambda}) \|_{L^p} \| \dot{\Delta}_k \Lambda \u_{\lambda}\|^{p-1}_{L^p}
	\end{align}
	where we have used the fact
	\begin{align*}
		\int_{\R^d} \u \cdot \nabla \Lambda \dot{\Delta}_k \u_{\lambda}  | \dot{\Delta}_k \Lambda \u_{\lambda} |^{p-2} \cdot \dot{\Delta}_k \Lambda \u_{\lambda} \,dx
 =& \frac{1}{p} \int_{\R^d}  \u \cdot \nabla |\dot{\Delta}_k \Lambda \u_{\lambda} |^pdx\nn\\
  = & -\frac{1}{p} \int_{\R^d} \div \u |\dot{\Delta}_k \Lambda \u_{\lambda} |^pdx=0.
	\end{align*}

Thanks to the H\"oler inequality, we can infer from \eqref{du} that
	\begin{align}\label{duk}
		&\frac{1}{p} \frac{d}{dt} \| \dot{\Delta}_k \Lambda \u_{\lambda}\|_{L^p} + \lambda f(t) \| \dot{\Delta}_k \Lambda \u_{\lambda}\|_{L^p} + \frac{1}{2} \| \dot{\Delta}_k \Lambda \u_{\lambda}\|_{L^p}\nn\\
 &\quad\le \|(([\u \cdot \nabla, \Lambda \dot{\Delta}_k \mathbb{P}] \u)_{\lambda},  \dot{\Delta}_k \Lambda \mathbb{P} \div \G_{\lambda}) \|_{L^p}.
	\end{align}
In terms of $\u$, the function $\G$ satisfies the following  heat type equation
\begin{align}\label{gfg}
\partial_t \G - \Delta \G =  \nabla \Delta^{-1} \Big(-\frac{1}{2} \mathbf{N}^1  + \frac{1}{2} \mathbb{P} \div \G + \frac{1}{4} \u\Big) - \mathbb{Q} \mathbf{N}^2 + \q \tau.
\end{align}
Applying $\ddk\exp\{- \lambda \int_{0}^{t} f(t')dt'\}$ to the above equation leads to
\begin{align}\label{w_k}
		\partial_t \dot{\Delta}_k \G_{\lambda} - \Delta \dot{\Delta}_k \G_{\lambda} =  \dot{\Delta}_k \nabla \Delta^{-1} (-\frac{1}{2} \mathbf{N}^1  + \frac{1}{2} \mathbb{P} \div \G + \frac{1}{4} \u)_{\lambda}  - \dot{\Delta}_k \mathbb{Q} \mathbf{N}^2_{\lambda} + \ddk \q \tau_{\lambda}.
	\end{align}

Thus, we get   by a similar derivation of \eqref{ptbar_k} that
	\begin{align}\label{wk}
		&\frac{1}{p} \frac{d}{dt} \| \dot{\Delta}_k \G_{\lambda}\|_{L^p} + \lambda f(t) \| \dot{\Delta}_k \G_{\lambda}\|_{L^p} + 2^{2k} \| \dot{\Delta}_k \G_{\lambda}\|_{L^p} \nn\\
		&\quad\lesssim 2^{-k} \| \dot{\Delta}_k \mathbf{N}^1_{\lambda} \|_{L^p} + 2^{-k} \| \dot{\Delta}_k \u_{\lambda}\|_{L^p} + \| \dot{\Delta}_k \G_{\lambda} \|_{L^p} + \| \dot{\Delta}_k \mathbf{N}^2_{\lambda} \|_{L^p} + \| \ddk \q \tau_{\lambda}\|_{L^p}.
	\end{align}
Let $\beta>0$ be a small fixed constant which  will be determined later, multiplying \eqref{duk} by $  2\beta$ and then adding to \eqref{wk}, we can get
	\begin{align}\label{duw_k}
		&\frac{1}{p} \frac{d}{dt} ( 2\beta \| \dot{\Delta}_k \Lambda \u_{\lambda}\|_{L^p} + \| \dot{\Delta}_k \G_{\lambda}\|_{L^p} ) +  \beta \|\Lambda \dot{\Delta}_k  \u_{\lambda}\|_{L^p} + 2^{2k} \| \dot{\Delta}_k \G_{\lambda}\|_{L^p} \nn\\
		&\qquad +  \lambda f(t) (2\beta \| \dot{\Delta}_k \Lambda \u_{\lambda}\|_{L^p} + \| \dot{\Delta}_k \G_{\lambda}\|_{L^p})\nn\\
 &\quad\le C(\beta 2^{2k}+1) \| \dot{\Delta}_k \G_{\lambda}\|_{L^p} + C 2^{-k} \| \dot{\Delta}_k  \u_{\lambda}\|_{L^p} \nn\\
		&\quad \quad + C(2^{-k} \| \dot{\Delta}_k \mathbf{N}^1_{\lambda} \|_{L^p}  + \| \dot{\Delta}_k \mathbf{N}^2_{\lambda} \|_{L^p} + \| ([\u \cdot \nabla, \Lambda \dot{\Delta}_k \mathbb{P}] \u)_{\lambda}\|_{L^p} + \| \ddk \q \tau_{\lambda}\|_{L^p}).
	\end{align}
	It is easy to see that
	\begin{align*}
		C2^{-k} \| \dot{\Delta}_k \u_{\lambda} \|_{L^p} \le C2^{-2k_0} \| \Lambda
		 \dot{\Delta}_k \u_{\lambda} \|_{L^p}       \qquad   \text{for all}\  k\ge k_0 -1.
	\end{align*}

	Choosing $ k_0$ is large enough and $ \beta$ is sufficiently small such that $ \frac{\beta}{2} \ge C2^{-2k_0}$ and  $ \frac{1}{2} \ge C\beta $, we infer from \eqref{duw_k} that
	\begin{align}\label{duwk}
		&\frac{1}{p} \frac{d}{dt} ( 2\beta \| \dot{\Delta}_k \Lambda \u_{\lambda}\|_{L^p} + \| \dot{\Delta}_k \G_{\lambda}\|_{L^p} )  +  \beta \|\Lambda \dot{\Delta}_k  \u_{\lambda}\|_{L^p} + 2^{2k} \| \dot{\Delta}_k \G_{\lambda}\|_{L^p}\nn\\
&\qquad+  \lambda f(t) (2\beta \| \dot{\Delta}_k \Lambda \u_{\lambda}\|_{L^p} + \| \dot{\Delta}_k \G_{\lambda}\|_{L^p})\nn\\
		&\quad \le   C(2^{-k} \| \dot{\Delta}_k \mathbf{N}^1_{\lambda} \|_{L^p}  + \| \dot{\Delta}_k \mathbf{N}^2_{\lambda} \|_{L^p} + \| ([\u \cdot \nabla, \Lambda \dot{\Delta}_k \mathbb{P}] \u)_{\lambda}\|_{L^p} +\| \ddk \q \tau_{\lambda}\|_{L^p}).
	\end{align}
	On the one hand, there holds
	\begin{align*}
		2\beta \| \dot{\Delta}_k \Lambda \u_{\lambda}\|_{L^p} + \| \dot{\Delta}_k \G_{\lambda}\|_{L^p} \lesssim \| \dot{\Delta}_k \Lambda \u_{\lambda}\|_{L^p} + \| \dot{\Delta}_k \mathbb{Q}\tau_{\lambda}\|_{L^p}.
	\end{align*}
	On the other hand, due to $ \frac{\beta}{2} \ge C2^{-2k_0}$, we have
	 \begin{align*}
	 	2\beta \| \dot{\Delta}_k \Lambda \u_{\lambda}\|_{L^p} + \| \dot{\Delta}_k \G_{\lambda}\|_{L^p} \gtrsim  \| \dot{\Delta}_k \Lambda \u_{\lambda}\|_{L^p} + \| \dot{\Delta}_k \mathbb{Q}\tau_{\lambda}\|_{L^p}.
	 \end{align*}

	Hence, by using the fact $ \|\dot{\Delta}_k \delta \|_{L^q} \approx \|\dot{\Delta}_k \mathbb{Q} \tau \|_{L^q} (q>1)$, we can get
$$ 2\beta \| \dot{\Delta}_k \Lambda \u_{\lambda}\|_{L^p} + \| \dot{\Delta}_k \G_{\lambda}\|_{L^p} \approx \| \dot{\Delta}_k \Lambda \u_{\lambda}\|_{L^p} + \| \dot{\Delta}_k \delta_{\lambda}\|_{L^p},$$
	 from which and multiplying the equation \eqref{duwk} by $ 2^{\frac{d}{p}k}$, taking the summation of $ k \ge k_0 $, and then integrating over $ [0, t]$, we can finally  get
	\begin{align}\label{le_uqt_h}
		&\| (\Lambda \u_{\lambda}, \delta_{\lambda})\|^h_{\widetilde{L}^{\infty}_{t}(\dot B^{\frac{d}{p}}_{p,1})} +   \lambda \| (\Lambda \u_{\lambda} , \delta_{\lambda} )\|^h_{{L}^{1}_{t, f}(\dot B^{\frac{d}{p}}_{p,1})}   + \| \u_{\lambda} \|^h_{{L}^{1}_{t}(\dot B^{\frac{d}{p}+1}_{p,1})} + \|  \delta_{\lambda} \|^h_{{L}^{1}_{t}(\dot B^{\frac{d}{p} + 2}_{p,1})} \nn\\
		&\quad\le C \| (\Lambda \mathbf{u}_0,\mathbb{Q}\tau_0)\|^h_{\dot B^{\frac{d}{p}}_{p,1}} + C\| \mathbf{N}^1_{\lambda} \|^h_{\tilde{L}^1_{t} (\dot{B}^{\frac{d}{p} -1}_{p,1} )}   \nn\\
		&\qquad +C \| \mathbf{N}^2_{\lambda} \|^h_{\tilde{L}^1_{t} (\dot{B}^{\frac{d}{p}}_{p,1} )}+ C \int_{0}^{t} \sum_{k \ge k_0 -1} 2^{\frac{d}{p}k} \| ([\u \cdot \nabla, \Lambda \dot{\Delta}_k \mathbb{P}]\u)_{\lambda}\|_{L^p} \,dt'.
	\end{align}
 To bound the last term in \eqref{le_uqt_h}, we claim the following important
commutator estimate, which we shall postpone its proof at the end of this lemma.

\noindent { $\mathrm{\mathbf{Claim:}}$}
Let $ k_0$ be a fixed positive constant and $ p$ satisfy $ 1 \le p \le 2d$, there holds
	\begin{align}\label{R_h}
	\sum_{k \ge k_0 -1} 2^{\frac{d}{p}k} \| [\u \cdot \nabla, \Lambda \dot{\Delta}_k \mathbb{P}]\u\|_{L^p} \le C \| \nabla \u\|^2_{\dot B^{\frac{d}{p}}_{p,1}}.
	\end{align}

\vskip .2in
\noindent Now, for the last term in \eqref{le_uqt_h}, it follows from \eqref{R_h} directly that
		\begin{align}\label{last_h}
		&\int_{0}^{t} \sum_{k \ge k_0 -1} 2^{\frac{d}{p}k} \| ([\u \cdot \nabla, \Lambda \dot{\Delta}_k \mathbb{P}]\u)_{\lambda}\|_{L^p} dt' \nn\\
		&\quad \lesssim (\| \u\|^\ell_{\widetilde{L}^{\infty}_{t} (\dot{B}^{\frac{d}{p} -1}_{p,1} )} + \| \u\|^h_{\widetilde{L}^{\infty}_{t} (\dot{B}^{\frac{d}{p} +1}_{p,1} )}) (\| \u_{\lambda}\|^\ell_{{L}^{1}_{t} (\dot{B}^{\frac{d}{p} +1}_{p,1} )} + \| \u_{\lambda}\|^h_{{L}^{1}_{t} (\dot{B}^{\frac{d}{p} +1}_{p,1} )}).
  		\end{align}
		For the term $ \mathbf{N}^1_{\lambda}$, applying Lemma \ref{le2.6} gives
		\begin{align}\label{1_h}
			\| \mathbf{N}^1 \|^h_{\dot{B}^{\frac{d}{p} -1}_{p,1} } \lesssim (\| \u  \|^\ell_{\dot{B}^{\frac{d}{p} -1}_{p,1} } + \| \u  \|^h_{\dot{B}^{\frac{d}{p} +1}_{p,1} }) (\| \u  \|^\ell_{\dot{B}^{\frac{d}{p} +1}_{p,1} } + \| \u  \|^h_{\dot{B}^{\frac{d}{p} +1}_{p,1} }).
		\end{align}
		Hence,
		\begin{align}\label{first_h}
			\| \mathbf{N}^1_{\lambda} \|^h_{\tilde{L}^1_{t} (\dot{B}^{\frac{d}{p} -1}_{p,1} )} \lesssim (\| \u  \|^\ell_{\widetilde{L}^{\infty}_t(\dot{B}^{\frac{d}{p} -1}_{p,1} )} + \| \u  \|^h_{\widetilde{L}^{\infty}_t(\dot{B}^{\frac{d}{p} +1}_{p,1} )}) (\| \u_{\lambda}  \|^\ell_{{L}^{1}_t(\dot{B}^{\frac{d}{p} +1}_{p,1} )} + \| \u_{\lambda}  \|^h_{{L}^{1}_t(\dot{B}^{\frac{d}{p} +1}_{p,1} )}).
		\end{align}
		Finally, for the term $ \mathbf{N}^2_{\lambda}$, applying Lemma \ref{bernstein} and Lemma \ref{le2.6} gives
		\begin{align}\label{21_h}
			\| \u \cdot \nabla \tau \|^h_{\dot{B}^{\frac{d}{p} }_{p,1} } \lesssim \| \u\|_{\dot{B}^{\frac{d}{p} }_{p,1}} \|  \nabla \tau \|_{\dot{B}^{\frac{d}{p} }_{p,1} } \lesssim (\| \u\|^\ell_{\dot{B}^{\frac{d}{p} -1}_{p,1}} + \| \u\|^h_{\dot{B}^{\frac{d}{p} +1}_{p,1}})\| \tau\|_{\dot{B}^{\frac{d}{p} +1}_{p,1}}
		\end{align}
		and
		\begin{align}\label{22_h}
			\| g_{\alpha}(\tau, \nabla \u) \|^h_{\dot{B}^{\frac{d}{p} }_{p,1} } \lesssim \| \u\|_{\dot{B}^{\frac{d}{p} }_{p,1}} \|   \tau \|_{\dot{B}^{\frac{d}{p} }_{p,1} } \lesssim (\| \u\|^\ell_{\dot{B}^{\frac{d}{p}}_{p,1}} + \| \u\|^h_{\dot{B}^{\frac{d}{p}}_{p,1}})\| \tau\|_{\dot{B}^{\frac{d}{p}}_{p,1}}
		\end{align}
		from  which and  \eqref{jiaquanguji}, we have
		\begin{align}\label{h2_h}
		&\|   \mathbf{N}^2_{\lambda} \|^h_{{L}^{1}_t ({\dot B^{\frac{d}{p}}_{p,1}})} \lesssim  \| \u_{\lambda} \|^\ell_{{L}^{1}_{t, f}(\dot B^{\frac{d}{p} - 1}_{p,1})}  +  \| \u_{\lambda} \|^h_{{L}^{1}_{t, f}(\dot B^{\frac{d}{p} + 1}_{p,1})} +\ea\eb.
		\end{align}
	Consequently, 	inserting \eqref{last_h},  \eqref{first_h} and \eqref{h2_h} into \eqref{le_uqt_h} we can arrive at \eqref{ut_h}.

Before ending the proof of the lemma, we have to  prove our claim \eqref{R_h}.
		We first use  Bony's decomposition to rewrite
		\begin{align}\label{remain}
		 &[\u \cdot \nabla , \Lambda \dot{\Delta}_k \mathbb{P}]\u = \u \cdot \nabla (\Lambda \dot{\Delta}_k \u) - \Lambda \dot{\Delta}_k \mathbb{P} (\u \cdot \nabla \u)\nn\\
		&\quad= \dot{T}_{u^j} \partial_j \Lambda \dot{\Delta}_k \u + \dot{T}_{\partial_j \Lambda \dot{\Delta}_k \u} \u^j + \dot{R}(\u^j, \partial_j \Lambda \dot{\Delta}_k \u) - \Lambda \dot{\Delta}_k \mathbb{P}( \dot{T}_{\u^j} \partial_j \u + \dot{T}_{\partial_j \u} \u^j + \dot{R}(\u^j, \partial_j \u))\nn\\
		&\quad= [\dot{T}_{\u_j}, \Lambda \dot{\Delta}_k \mathbb{P}]\partial_j \u + \dot{T}_{\partial_j \Lambda \dot{\Delta}_k \u} \u^j - \Lambda \dot{\Delta_k} \mathbb{P} \dot{T}_{\partial_j \u} \u^j + \dot{R}(\u^j, \partial_j \Lambda \dot{\Delta}_k \u) - \Lambda \dot{\Delta}_k \mathbb{P} \dot{R}(\u^j, \partial_j \u).
		\end{align}
		Next we need to bound the right terms of \eqref{remain}. For the first term, on the basis of the  definition of operator $ \dot{T}$ and $ \div \u = 0$, we can get
		\begin{align}\label{remain_1}
		&[\dot{T}_{\u_j}, \Lambda \dot{\Delta}_k \mathbb{P}]\partial_j \u = \dot{T}_{u^j} \partial_j \Lambda \dot{\Delta}_k \u - \Lambda \dot{\Delta}_k \mathbb{P} \dot{T}_{\u^j} \partial_j \u \nn\\
		& \quad= \sum_{k' \in \Z} \dot{S}_{k' - 1} \u^j \dot{\Delta}_{k'} (\Lambda \partial_j \dot{\Delta}_{k} \u) - \sum_{k' \in \Z} \Lambda \dot{\Delta}_{k} (\dot{S}_{k' - 1} \u^j \dot{\Delta}_{k'} \partial_j \u) + \sum_{k' \in \Z} \Lambda \dot{\Delta}_{k} \mathbb{Q} (\dot{S}_{k' - 1} \u^j \dot{\Delta}_{k'} \partial_j \u) \nn\\
		&\quad = \sum_{| k-k'| \le 4}[\dot{S}_{k' - 1} \u^j, \dot{\Delta}_{k} \Lambda]\partial_j \dot{\Delta}_{k'} \u - \sum_{| k-k'| \le 4}[\dot{S}_{k' - 1} \u^j, \dot{\Delta}_{k} \Lambda \mathbb{Q}]\partial_j \dot{\Delta}_{k'} \u \nn\\
&\quad\stackrel{\mathrm{def}}{=} I_{k}^1 + I_{k}^2.
		\end{align}
		Since $ \dot{\Delta}_{k} f = 2^{kd} \int_{\R^d} \varphi(2^k y)f(x-y) dy$, then there exists a function $ \widetilde{\varphi}$ with a compact support set such that $ \dot{\Delta}_{k} \Lambda f = 2^{k(d +1)} \int_{\R^d} \widetilde{\varphi}(2^k y)f(x-y) dy$. The operator $ \mathbb{Q}$ is a zero order Fourier multiplier, then there exists a function $ \bar{\varphi}$ with a compact support set such that $$ \dot{\Delta}_{k} \Lambda \mathbb{Q}f = 2^{k(d +1)} \int_{\R^d} \bar{\varphi}(2^k y)f(x-y) dy.$$
Hence, we have
		\begin{align}\label{Ik1}
		\|I_{k}^1\|_{L^p} &= \Big\|  \sum_{| k-k'| \le 4} 2^{k(d +1)} \int_{\R^d} \widetilde{\varphi}(2^k(x-y)) [\dot{S}_{k' - 1} \u^j (x) - \dot{S}_{k' - 1} \u^j(y)]\partial_j \dot{\Delta}_{k'} \u(y)dy \Big\|_{L^p}\nn\\
		&\le  \Big\|  \sum_{| k-k'| \le 4} 2^{k(d +1)} \int_{\R^d} \widetilde{\varphi}(2^k(x-y)) [\int_{0}^{1}(x-y) \nabla \dot{S}_{k' - 1} \u^j (y+\alpha(x-y)) d\alpha] \partial_j \dot{\Delta}_{k'} \u(y)dy \Big\|_{L^p}\nn\\
		&\le \| \nabla \u\|_{L^{\infty}} \sum_{| k-k'| \le 4} 2^{k(d +1)}  \Big\| \int_{\R^d} \widetilde{\varphi}(2^k(x-y)) (x-y) \partial_j \dot{\Delta}_{k'} \u(y)dy \Big\|_{L^p}\nn\\
		& \le \| \nabla \u\|_{L^{\infty}} \sum_{| k-k'| \le 4} 2^{k(d +1)}  \Big\{ \int_{\R^d} \widetilde{\varphi}(2^ky)ydy \Big\} \| \partial_j \dot{\Delta}_{k'} \u(y)\|_{L^p}\nn\\
		& \le C \| \nabla \u\|_{L^{\infty}}  \| \partial_j \dot{\Delta}_{k'} \u\|_{L^p}
		\end{align}
		where we have used the fact $ 2^{k(d +1)} \int_{\R^d} \widetilde{\varphi}(2^ky)ydy \le C$. Similarly, we can also get $$ \|I_{k}^2\|_{L^p} \le C \| \nabla \u\|_{L^{\infty}}  \| \partial_j \dot{\Delta}_{k'} \u\|_{L^p}$$
from which we get
		\begin{align}\label{1}
		\sum_{k \ge k_0 -1} 2^{k\frac{d}{p}} \| [\dot{T}_{\u_j}, \Lambda \dot{\Delta}_k \mathbb{P}]\partial_j \u \|_{L^p} \lesssim& \| \nabla \u\|_{L^{\infty}} \| \nabla \u\|_{\dot{B}^{\frac{d}{p}}_{p,1}}
\lesssim \| \nabla \u\|^2_{\dot{B}^{\frac{d}{p}}_{p,1}}.
		\end{align}
		According to the Lemma \ref{le2.5} and Lemma \ref{embedding},  there holds
		\begin{align}\label{2}
		\sum_{k \ge k_0 -1} 2^{k\frac{d}{p}} \| \dot{T}_{\partial_j \Lambda \dot{\Delta}_k \u} \u^j\|_{L^p} \lesssim \| \dot{T}_{\partial_j \Lambda  \u} \u^j\|_{\dot{B}^{\frac{d}{p}}_{p,1}} \lesssim \| \nabla^2 \u\|_{\dot{B}^{-1}_{\infty,1}} \| \u \|_{\dot{B}^{\frac{d}{p} + 1}_{p,1}}
		\lesssim \| \nabla \u\|^2_{\dot{B}^{\frac{d}{p}}_{p,1}}.
		\end{align}
		Similarly, we have
		\begin{align}\label{3}
		\sum_{k \ge k_0 -1} 2^{k\frac{d}{p}} \| \Lambda \dot{\Delta_k} \mathbb{P} \dot{T}_{\partial_j \u} \u^j \|_{L^p} &\lesssim  \| \Lambda  \dot{T}_{\partial_j \u} \u^j \|_{\dot{B}^{\frac{d}{p}}_{p,1}}  \lesssim \| \dot{T}_{\partial_j \u} \u^j \|_{\dot{B}^{\frac{d}{p} + 1}_{p,1}}\nn\\& \lesssim \| \nabla \u \|_{\dot{B}^{0}_{\infty, 1}}  \|  \u \|_{\dot{B}^{\frac{d}{p} + 1}_{p,1}}
		 \lesssim \| \nabla \u\|^2_{\dot{B}^{\frac{d}{p}}_{p,1}};
		\end{align}
		\begin{align}\label{4}
		\sum_{k \ge k_0 -1} 2^{k\frac{d}{p}} \| \dot{R}(\u^j, \partial_j \Lambda \dot{\Delta}_k \u) \|_{L^p} \lesssim \| \dot{R}(\u^j, \partial_j \Lambda  \u) \|_{\dot{B}^{\frac{d}{p}}_{p,1}} \lesssim \|  \u \|_{\dot{B}^{\frac{d}{p} + 1}_{p,1}} \| \nabla^2 \u \|_{\dot{B}^{-1}_{\infty, 1}}
		 \lesssim \| \nabla \u\|^2_{\dot{B}^{\frac{d}{p}}_{p,1}};
		\end{align}
		and
		\begin{align}\label{5}
		\sum_{k \ge k_0 -1} 2^{k\frac{d}{p}} \| \Lambda \dot{\Delta}_k \mathbb{P} \dot{R}(\u^j, \partial_j \u)\|_{L^p} \lesssim \|    \dot{R}(\u^j, \partial_j \u)\|_{\dot{B}^{\frac{d}{p} + 1}_{p,1}} \lesssim \| \nabla \u \|_{\dot{B}^{0}_{\infty, 1}}  \|  \u \|_{\dot{B}^{\frac{d}{p} + 1}_{p,1}}
		\lesssim \| \nabla \u\|^2_{\dot{B}^{\frac{d}{p}}_{p,1}}.
		\end{align}
		Combining \eqref{1}$ -$\eqref{5}, we can get \eqref{R_h}.	 Consequently, we complete the proof of the lemma.
	\end{proof}
\end{lemma}

\subsection{Completion the proof of Theorem \ref{theorem1.1}}
In this subsection, we use the continuity argument to complete the proof of Theorem \ref{theorem1.1}. At first, we define $ T^*$ by
\begin{align}\label{T*}
	&T^*  \stackrel{\mathrm{def}}{=} \sup \bigg \{ t\in[ 0, T^*) : \| \u\|^\ell_{\widetilde{L}^{\infty}_t ({\dot B^{\frac{d}{p} - 1}_{p,1}})} +\| \u\|^h_{\widetilde{L}^{\infty}_t ({\dot B^{\frac{d}{p} + 1}_{p,1}})} + \| \delta\|^\ell_{\widetilde{L}^{\infty}_t ({\dot B^{\frac{d}{p} }_{p,1}})} +  \| \delta\|^h_{\widetilde{L}^{\infty}_t ({\dot B^{\frac{d}{p} }_{p,1}})} + \| \mathbb{P} \bar{\tau} \|_{\widetilde{L}^{\infty}_t(\dot B^{\frac{d}{p} }_{p,1})} \nn\\
	&\quad + \| \u \|^\ell_{{L}^{1}_t ({\dot B^{\frac{d}{p} + 1}_{p,1}})} + \| \u \|^h_{{L}^{1}_t ({\dot B^{\frac{d}{p} + 1}_{p,1}})} +  \| \delta \|^\ell_{{L}^{1}_t ({\dot B^{\frac{d}{p} }_{p,1}})} +  \| \delta \|^h_{{L}^{1}_t ({\dot B^{\frac{d}{p} + 2}_{p,1}})} +  \| \mathbb{P} \bar{\tau} \|_{{L}^{1}_t(\dot B^{\frac{d}{p}  }_{p,1})}+  \| \mathbb{P} \bar{\tau} \|_{{L}^{1}_t(\dot B^{\frac{d}{p} +2 }_{p,1})} \le c_0     \bigg \}.
\end{align}
Where $0 < c_0 \ll 1$ will be determined later. According to the local well-posedness of system \eqref{reformulate}, we shall prove that $T^* > 0$ can be extended to $T^* = \infty$.\\
 Then multiplying \eqref{ut_low} and \eqref{ut_h} by an appropriate constant$ (2C+1)$ and adding to \eqref{pt}, finally, we choose $ \lambda \ge 10C$ , we can get
\begin{align}\label{last_1}
	&  \| \u_{\lambda} \|^\ell_{\widetilde{L}^{\infty}_{t}(\dot B^{\frac{d}{p} - 1}_{p,1})} + \|  \delta_{\lambda}\|^\ell_{\widetilde{L}^{\infty}_{t}(\dot B^{\frac{d}{p}}_{p,1})} + \| \mathbb{P} \bar{\tau}_{\lambda} \|_{\widetilde{L}^{\infty}_{t}(\dot B^{\frac{d}{p}}_{p,1})} + \| (\Lambda \u_{\lambda}, \delta_{\lambda})\|^h_{\widetilde{L}^{\infty}_{t}(\dot B^{\frac{d}{p}}_{p,1})} + \frac{\lambda}{2} \| \u_{\lambda}  \|^\ell_{{L}^{1}_{t, f}(\dot B^{\frac{d}{p} - 1}_{p,1})} \nn\\
	&\qquad + \frac{\lambda}{2} \|  \delta_{\lambda} \|^\ell_{{L}^{1}_{t, f}(\dot B^{\frac{d}{p}}_{p,1})} + \lambda \| \mathbb{P} \bar{\tau}_{\lambda} \|_{{L}^{1}_{t, f}(\dot B^{\frac{d}{p} }_{p,1})} + \frac{\lambda}{2} \| \u_{\lambda} \|^h_{{L}^{1}_{t, f}(\dot B^{\frac{d}{p}+1}_{p,1})}  +  \lambda \|  \delta_{\lambda} \|^h_{{L}^{1}_{t, f}(\dot B^{\frac{d}{p}}_{p,1})}+  \| \u_{\lambda}\|^\ell_{{L}^{1}_{t}(\dot B^{\frac{d}{p} + 1}_{p,1} )} \nn\\
&\qquad  +\| \delta_{\lambda}\|^\ell_{{L}^{1}_{t}(\dot B^{\frac{d}{p}}_{p,1} )} +  \| \mathbb{P} \bar{\tau}_{\lambda} \|_{{L}^{1}_{t}(\dot B^{\frac{d}{p} }_{p,1} )} + \| \mathbb{P} \bar{\tau}_{\lambda} \|_{{L}^{1}_{t}(\dot B^{\frac{d}{p} +2 }_{p,1} )}+  \| \u_{\lambda} \|^h_{{L}^{1}_{t}(\dot B^{\frac{d}{p}+1}_{p,1})}
	+ \|  \delta_{\lambda} \|^h_{{L}^{1}_{t}(\dot B^{\frac{d}{p} + 2}_{p,1})}\nn\\
	& \quad\le C(\| \u_0 \|^\ell_{\dot{B}^{\frac{d}{p}-1}_{p,1}} +\|  \delta_0\|^\ell_{\dot{B}^{\frac{d}{p}}_{p,1}}  +  \| (\Lambda \mathbf{u}_0,\mathbb{Q}\tau_0)\|^h_{\dot B^{\frac{d}{p}}_{p,1}} ) \nn\\
	  &\qquad+ C(\ea + \| \delta\|^\ell_{\widetilde{L}^{\infty}_t ({\dot B^{\frac{d}{p} }_{p,1}})} +  \| \delta\|^h_{\widetilde{L}^{\infty}_t ({\dot B^{\frac{d}{p} }_{p,1}})} + \| \mathbb{P} \bar{\tau} \|_{\widetilde{L}^{\infty}_t(\dot B^{\frac{d}{p} }_{p,1})} )\eb.
\end{align}
For any $t \le T^*$, we can deduce from\eqref{T*} that
\begin{align}\label{last_2}
		&  \| \u_{\lambda} \|^\ell_{\widetilde{L}^{\infty}_{t}(\dot B^{\frac{d}{p} - 1}_{p,1})} + \|  \delta_{\lambda}\|^\ell_{\widetilde{L}^{\infty}_{t}(\dot B^{\frac{d}{p}}_{p,1})}  + \| \mathbb{P} \bar{\tau}_{\lambda} \|_{\widetilde{L}^{\infty}_{t}(\dot B^{\frac{d}{p} }_{p,1})} + \| (\Lambda \u_{\lambda}, \delta_{\lambda})\|^h_{\widetilde{L}^{\infty}_{t}(\dot B^{\frac{d}{p}}_{p,1})}+  \| \u_{\lambda}\|^\ell_{{L}^{1}_{t}(\dot B^{\frac{d}{p} + 1}_{p,1} )} \nn\\
		&\qquad + \|  \delta_{\lambda}\|^\ell_{{L}^{1}_{t}(\dot B^{\frac{d}{p}}_{p,1} )}   +  \| \mathbb{P} \bar{\tau}_{\lambda} \|_{{L}^{1}_{t}(\dot B^{\frac{d}{p} }_{p,1} )} + \| \mathbb{P} \bar{\tau}_{\lambda} \|_{{L}^{1}_{t}(\dot B^{\frac{d}{p} + 2}_{p,1} )}
	+  \| \u_{\lambda} \|^h_{{L}^{1}_{t}(\dot B^{\frac{d}{p}+1}_{p,1})}
	+ \|  \delta_{\lambda} \|^h_{{L}^{1}_{t}(\dot B^{\frac{d}{p} + 2}_{p,1})} \nn\\
	&\qquad+ \frac{\lambda}{4} \Big( \| \u_{\lambda}  \|^\ell_{{L}^{1}_{t, f}(\dot B^{\frac{d}{p} - 1}_{p,1})} + \|  \delta_{\lambda} \|_{{L}^{1}_{t, f}(\dot B^{\frac{d}{p}}_{p,1})} + \| \mathbb{P} \bar{\tau}_{\lambda} \|_{{L}^{1}_{t, f}(\dot B^{\frac{d}{p}}_{p,1})} +  \| \u_{\lambda} \|^h_{{L}^{1}_{t, f}(\dot B^{\frac{d}{p}+1}_{p,1})} \Big)  \nn\\
		& \quad\le C(\| \u_0 \|^\ell_{\dot{B}^{\frac{d}{p}-1}_{p,1}} +\|  \delta_0\|^\ell_{\dot{B}^{\frac{d}{p}}_{p,1}}  +  \| (\Lambda \mathbf{u}_0,\mathbb{Q}\tau_0)\|^h_{\dot B^{\frac{d}{p}}_{p,1}} ).
\end{align}
From the definition \ref{chemin} of Chemin-Lerner norm,
\begin{align}\label{last_3}
	\| \u_{\lambda} \|_{\widetilde{L}^{\infty}_{t}(\dot B^{\frac{d}{p} - 1}_{p,1})}
	& = \sum_{k \in \Z} 2^{k(\frac{d}{p} - 1)} \sup\limits_{0 \le t \le T^*} \Big \{  \| \dot{\Delta}_k \u\|_{L^p(\R^d)} \exp^{-\lambda \int_{0}^{t} f(t') \,dt'} \Big \}\nn\\
	& \ge \sum_{k \in \Z} 2^{k(\frac{d}{p} - 1)} \sup\limits_{0 \le t \le T^*} \| \dot{\Delta}_k \u\|_{L^p(\R^d)} \cdot \exp^{-\lambda \int_{0}^{T^*} f(t') \,dt'} \nn\\
	& = \| \u \|_{\widetilde{L}^{\infty}_{t}(\dot B^{\frac{d}{p} - 1}_{p,1})} \cdot \exp\Big({-\lambda \int_{0}^{T^*} f(t') \,dt'}\Big).
\end{align}
Similarly,
\begin{align}\label{last_4}
\| \u_{\lambda} \|_{{L}^{1}_{t}(\dot B^{\frac{d}{p} + 1}_{p,1})} = \| \u \|_{{L}^{1}_{t}(\dot B^{\frac{d}{p} + 1}_{p,1})} \cdot \exp\Big({-\lambda \int_{0}^{T^*} f(t') \,dt'}\Big).
\end{align}
As a consequence, we can deduce from \eqref{last_2} that
\begin{align}\label{last_5}
  &\| \u\|^\ell_{\widetilde{L}^{\infty}_{t}(\dot B^{\frac{d}{p} - 1}_{p,1})} + \|\delta\|^\ell_{\widetilde{L}^{\infty}_{t}(\dot B^{\frac{d}{p} }_{p,1})}+ \| \mathbb{P} \bar{\tau} \|_{\widetilde{L}^{\infty}_{t}(\dot B^{\frac{d}{p} }_{p,1})} + \| (\Lambda \u, \delta)\|^h_{\widetilde{L}^{\infty}_{t}(\dot B^{\frac{d}{p}}_{p,1})} +  \| \u\|^\ell_{{L}^{1}_{t}(\dot B^{\frac{d}{p} + 1}_{p,1} )}   \nn\\
  &\qquad+ \|  \delta\|^\ell_{{L}^{1}_{t}(\dot B^{\frac{d}{p}}_{p,1} )}  + \| \mathbb{P} \bar{\tau} \|_{{L}^{1}_{t}(\dot B^{\frac{d}{p} }_{p,1} )}+  \| \mathbb{P} \bar{\tau} \|_{{L}^{1}_{t}(\dot B^{\frac{d}{p} + 2}_{p,1} )}
	+  \| \u \|^h_{{L}^{1}_{t}(\dot B^{\frac{d}{p}+1}_{p,1})}
	+ \|  \delta \|^h_{{L}^{1}_{t}(\dot B^{\frac{d}{p} + 2}_{p,1})}\nn\\
 &\le C(\| \u_0 \|^\ell_{\dot{B}^{\frac{d}{p}-1}_{p,1}} + \|\delta_0 \|^\ell_{\dot{B}^{\frac{d}{p}}_{p,1}} +  \| (\Lambda \mathbf{u}_0,\mathbb{Q}\tau_0)\|^h_{\dot B^{\frac{d}{p}}_{p,1}} ) \cdot \exp\Big({\lambda \int_{0}^{T^*} f(t') \,dt'}\Big).
\end{align}
According to the Lemma \ref{heat_kernel} and the definition of $ f(t)$, we have
\begin{align}\label{last_6}
	 \int_{0}^{T^*} f(t') \,dt' &= \int_{0}^{T^*} \Big(\| \mathbb{P} \tau_F\|_{\dot{B}^{\frac{d}{p}+2}_{p,1}} + \| \mathbb{P} \tau_F\|_{\dot{B}^{\frac{d}{p}}_{p,1}} \Big)dt'\nn\\
	 & \le C\| \mathbb{P} \tau_0\|_{\dot{B}^{\frac{d}{p}}_{p,1}}.
\end{align}

Hence, we conclude from \eqref{last_5} and \eqref{last_6}, for any $ 0 \le t \le T^*$,  that
\begin{align}\label{last_7}
	 &\| \u\|^\ell_{\widetilde{L}^{\infty}_{t}(\dot B^{\frac{d}{p} - 1}_{p,1})} + \|\delta\|^\ell_{\widetilde{L}^{\infty}_{t}(\dot B^{\frac{d}{p} }_{p,1})}+ \| \mathbb{P} \bar{\tau} \|_{\widetilde{L}^{\infty}_{t}(\dot B^{\frac{d}{p} }_{p,1})} + \| (\Lambda \u, \delta)\|^h_{\widetilde{L}^{\infty}_{t}(\dot B^{\frac{d}{p}}_{p,1})} +  \| \u\|^\ell_{{L}^{1}_{t}(\dot B^{\frac{d}{p} + 1}_{p,1} )}   \nn\\
	&\qquad+ \|  \delta\|^\ell_{{L}^{1}_{t}(\dot B^{\frac{d}{p}}_{p,1} )}  + \| \mathbb{P} \bar{\tau} \|_{{L}^{1}_{t}(\dot B^{\frac{d}{p} }_{p,1} )}+  \| \mathbb{P} \bar{\tau} \|_{{L}^{1}_{t}(\dot B^{\frac{d}{p} + 2}_{p,1} )}
	+  \| \u \|^h_{{L}^{1}_{t}(\dot B^{\frac{d}{p}+1}_{p,1})}
	+ \|  \delta \|^h_{{L}^{1}_{t}(\dot B^{\frac{d}{p} + 2}_{p,1})}\nn\\
	 &\quad\le C(\| \u_0 \|^\ell_{\dot{B}^{\frac{d}{p}-1}_{p,1}} + \|\delta_0 \|^\ell_{\dot{B}^{\frac{d}{p}}_{p,1}} +  \| (\Lambda \mathbf{u}_0,\mathbb{Q}\tau_0)\|^h_{\dot B^{\frac{d}{p}}_{p,1}} ) \cdot \exp\big (C\| \mathbb{P} \tau_0\|_{\dot{B}^{\frac{d}{p}}_{p,1}} \big).
\end{align}
Under the assumption of \eqref{X0}, we choose $ C_0$ large enough and $ c_0$ sufficiently small, there holds
\begin{align}\label{last}
	&\| \u\|^\ell_{\widetilde{L}^{\infty}_{t}(\dot B^{\frac{d}{p} - 1}_{p,1})} + \|\delta\|^\ell_{\widetilde{L}^{\infty}_{t}(\dot B^{\frac{d}{p} }_{p,1})}+ \| \mathbb{P} \bar{\tau} \|_{\widetilde{L}^{\infty}_{t}(\dot B^{\frac{d}{p} }_{p,1})} + \| (\Lambda \u, \delta)\|^h_{\widetilde{L}^{\infty}_{t}(\dot B^{\frac{d}{p}}_{p,1})} +  \| \u\|^\ell_{{L}^{1}_{t}(\dot B^{\frac{d}{p} + 1}_{p,1} )}   + \|  \delta\|^\ell_{{L}^{1}_{t}(\dot B^{\frac{d}{p}}_{p,1} )} \nn\\
	&\qquad + \| \mathbb{P} \bar{\tau} \|_{{L}^{1}_{t}(\dot B^{\frac{d}{p} }_{p,1} )}+  \| \mathbb{P} \bar{\tau} \|_{{L}^{1}_{t}(\dot B^{\frac{d}{p} + 2}_{p,1} )}
	+  \| \u \|^h_{{L}^{1}_{t}(\dot B^{\frac{d}{p}+1}_{p,1})}
	+ \|  \delta \|^h_{{L}^{1}_{t}(\dot B^{\frac{d}{p} + 2}_{p,1})}  \le \frac{c_0}{2}, \qquad \forall t \le T^*.
\end{align}
Which contradicts \eqref{T*}, hence we can extend $ T^* = \infty$. This completes the proof of Theorem \ref{theorem1.1}.

\hspace{15.7cm}$\square$

\section{The proof  of the global well-posedness of  Theorem \ref{theorem1.2}}
This section is devoted to the proof of the global well-posedness of  Theorem \ref{theorem1.2}. At first, we rewrite  system \eqref{2modle} into the following form:
\begin{eqnarray}\label{reformulate1}
\left\{\begin{aligned}
&\partial_t \u   +\nabla P -  \div \tau =  \mathbf{g}_1,\\
& \partial_t \tau  -  \Delta  \tau  - D(\u)  =  \mathbf{g}_2,
\end{aligned}\right.
\end{eqnarray}
where
\begin{align*}
&\mathbf{g}_1 = -\u\cdot\nabla \u;\quad
\mathbf{g}_2= - \u\cdot\nabla\tau - g_{\alpha}(\tau, \nabla \u).
\end{align*}
For clarity, the proof is divided into  four main subsections. The first three subsections are to make {\it a priori}
estimates of the solutions.
The last subsection is to complete the proof of  Theorem \ref{theorem1.2}.

For convenience, we denote $$ \widetilde{E}_1(t) \stackrel{\mathrm{def}}{=} \sup_{0 \le s \le t}\| (\u(s), \tau(s))\|^2_{H^3} + \int_{0}^{t} (\|\nabla \u(s) \|^2_{H^2} + \|\nabla \tau(s) \|^2_{H^3} )\, ds,$$
and $$\widetilde{E}_2(t) \stackrel{\mathrm{def}}{=} \sup_{0 \le s \le t} (1+s)^2 \| (\nabla \u(s), \nabla \tau(s))\|^2_{H^2} + \int_{0}^{t} (1+s)^2 (\|\nabla^2 \u(s) \|^2_{H^1} + \|\nabla^2 \tau(s) \|^2_{H^2}) \, ds .$$

\subsection{The dissipation of the viscoelastic stress tensor}
We show the first type of energy estimates in the following lemma which contains the dissipation estimate for $\tau.$
\begin{lemma}\label{L2}
		Let $(\u, \tau)$ be the solution of \eqref{reformulate1}, then there holds,  for any $t \ge 0$, that
\begin{align}\label{tauguji1}
			& \sup_{0 \le s \le t} \| (\u(s), \tau(s))\|^2_{H^3} + \int_{0}^{t}  \|\nabla \tau(s) \|^2_{H^3} \, ds \le C\widetilde{E}_1(0) + C\widetilde{E}_1(t) \widetilde{E}_2^{\frac{1}{2}} (t).
		\end{align}
\end{lemma}
	\begin{proof}  Firstly, applying operator $\nabla^k(k = 0,1,2,3)$ to \eqref{reformulate1}, and taking the $L^2$-inner product  with respect to $\nabla^k \u$ and $\nabla^k \tau$ in $\T^d$, respectively. Then adding them up gives rise to
		\begin{align}\label{H3_1}
			\frac{1}{2} \frac{d}{dt} \| (\u, \tau)\|^2_{H^3}  + \| \nabla \tau\|^2_{H^3}  &=  \sum_{k=0}^{3} \langle \nabla^k \mathbf{g}_1, \nabla^k \u \rangle + \sum_{k=0}^{3} \langle \nabla^k \mathbf{g}_2, \nabla^k \tau \rangle \stackrel{\mathrm{def}}{=} \widetilde{J}_1 + \widetilde{J}_2 
		\end{align}
		where we have used
\begin{align}\label{fact1}
			\langle \nabla^k \div \tau, \nabla^k \u \rangle + \langle \nabla^k D(\u), \nabla^k \tau \rangle = 0, \quad \text{and} \quad \langle \nabla^k \u, \nabla^k \nabla P \rangle = 0.
		\end{align}
For the first term $\widetilde{J}_1$,
		according to Leibniz's formula and the embedding inequality $ H^1(\T^d) \hookrightarrow L^3(\T^d)$, we have:
\begin{eqnarray}\label{L2_g1}
\left\{\begin{aligned}
\widetilde{J}_1 &= -\langle \nabla \u \cdot \nabla \u, \nabla \u  \rangle - 2 \langle \nabla \u \cdot \nabla^2 \u, \nabla^2 \u  \rangle - 2 \langle \nabla \u \cdot \nabla^3 \u, \nabla^3 \u  \rangle - \langle \nabla^2 \u \cdot \nabla^2 \u, \nabla^3 \u  \rangle\nn\\
			& \lesssim \| \nabla \u\|_{L^\infty} \| \nabla \u\|_{L^2} \| \nabla \u\|_{L^2} + \| \nabla \u\|_{L^\infty} \| \nabla^2 \u\|^2_{H^1} + \| \nabla^2 \u\|_{L^3} \| \nabla^2 \u\|_{L^6} \| \nabla^3 \u\|_{L^2}\nn\\
			&\lesssim \| \u\|^2_{H^3} \| \nabla^2 \u\|_{H^1},\quad\hbox{if $d = 3;$}\\
\widetilde{J}_1 & \lesssim \| \nabla \u\|_{L^\infty} \| \nabla \u\|_{L^2} \| \nabla \u\|_{L^2} + \| \nabla \u\|_{L^\infty} \| \nabla^2 \u\|^2_{H^1} + \| \nabla^2 \u\|_{L^4} \| \nabla^2 \u\|_{L^4} \| \nabla^3 \u\|_{L^2}\nn\\
			&\lesssim \| \u\|^2_{H^3} \| \nabla^2 \u\|_{H^1},\quad\hbox{if $d = 2,$}
\end{aligned}\right.
\end{eqnarray}
		where we have the fact $ \int_{\T^d} \u \cdot \nabla \nabla^j \u \cdot \nabla^j \u \, dx = 0 (j \ge 0)$.

To bound  the second term $\widetilde{J}_2$, we first
		decompose it	into another two terms
		\begin{align}\label{L2_g2}
			\widetilde{J}_2 &= - \sum_{k=0}^{3} \langle \nabla^k (\u \cdot \nabla \tau), \nabla^k \tau \rangle - \sum_{k=0}^{3} \langle \nabla^k (g_{\alpha}(\tau, \nabla \u)), \nabla^k \tau \rangle\nn\\
			& \stackrel{\mathrm{def}}{=} \widetilde{J}_{2,1} + \widetilde{J}_{2,2}.
		\end{align}
		For the term $\widetilde{J}_{2,1}$, by applying integration by parts, H\"older's inequality, and embedding relation, we obtain
		\begin{align}\label{J21}
			\text{for} \quad d = 3, \quad \widetilde{J}_{2,1} &= - \langle \nabla \u \cdot \nabla \tau, \nabla \tau  \rangle -  \langle \nabla \u \cdot \nabla^2 \tau, \nabla^2 \tau  \rangle - \langle \nabla^2 \u \cdot \nabla \tau, \nabla^2 \tau  \rangle -  \langle \nabla \u \cdot \nabla^3 \tau, \nabla^3 \tau  \rangle \nn\\
			&\quad \quad  - \langle \nabla^2 \u \cdot \nabla^2 \tau, \nabla^3 \tau  \rangle  - \langle \nabla^3 \u \cdot \nabla \tau, \nabla^3 \tau  \rangle\nn\\
			& \lesssim \| \nabla \u\|_{L^\infty} \| \nabla \tau\|^2_{H^2}  + \|\nabla \tau \|_{L^\infty} (\| \nabla^2 \u\|_{L^2} \| \nabla^2 \tau\|_{L^2} + \| \nabla^3 \u\|_{L^2} \| \nabla^2 \tau\|_{L^2}) \nn\\
			&\quad  + \|\nabla^2 \u \|_{L^3} \| \nabla^2 \tau\|_{L^6} \| \nabla^3 \tau\|_{L^2}\nn\\
			&\lesssim \| (\u, \tau)\|^2_{H^3} (\| \nabla^2 \u\|_{H^1} + \| \nabla^2 \tau\|_{H^2});\\
			\text{for} \quad d = 2, \quad \widetilde{J}_{2,1} & \lesssim \| \nabla \u\|_{L^\infty} \| \nabla \tau\|^2_{H^2}  + \|\nabla \tau \|_{L^\infty} (\| \nabla^2 \u\|_{L^2} \| \nabla^2 \tau\|_{L^2} + \| \nabla^3 \u\|_{L^2} \| \nabla^2 \tau\|_{L^2}) \nn\\
			&\quad  + \|\nabla^2 \u \|_{L^4} \| \nabla^2 \tau\|_{L^4} \| \nabla^3 \tau\|_{L^2}\nn\\
			&\lesssim \| (\u, \tau)\|^2_{H^3} (\| \nabla^2 \u\|_{H^1} + \| \nabla^2 \tau\|_{H^2}).
		\end{align}
		Similarly, there holds
		\begin{align}\label{J22}
			\widetilde{J}_{2,2} &= - \langle g_{\alpha}(\tau, \nabla \u), \tau \rangle + \sum_{k=1}^{3} \langle \nabla^{k-1} g_{\alpha}(\tau, \nabla \u), \nabla^{k+1}\tau \rangle\nn\\
			&\lesssim \| (\u, \tau)\|^2_{H^3} (\| \nabla^2 \u\|_{H^1} + \| \nabla^2 \tau\|_{H^2}).
		\end{align}
		Inserting \eqref{L2_g1}, \eqref{J21} and \eqref{J22} into \eqref{H3_1} gives rise to
		\begin{align}\label{energy1}
			\frac{1}{2} \frac{d}{dt} \| (\u, \tau)\|^2_{H^3}  + \| \nabla \tau\|^2_{H^3} \le C\| (\u, \tau)\|^2_{H^3} (\| \nabla^2 \u\|_{H^1} + \| \nabla^2 \tau\|_{H^2}).
		\end{align}
		Integrating the above equation over the interval from $0$ to $t$, we obtain
		\begin{align}\label{half}
			& \sup_{0 \le s \le t} \| (\u(s), \tau(s))\|^2_{H^3} + \int_{0}^{t}  \|\nabla \tau(s) \|^2_{H^3} \, ds \nn\\
			&\quad\le C\widetilde{E}_1(0) + C\int_{0}^{t} \| (\u, \tau)\|^2_{H^3} (\| \nabla^2 \u\|_{H^1} + \| \nabla^2 \tau\|_{H^2}) \, ds\nn\\
			&\quad\le C\widetilde{E}_1(0) + C\widetilde{E}_1(t) \int_{0}^{t} (1+s)^{-1} (1+s) (\| \nabla^2 \u\|_{H^1} + \| \nabla^2 \tau\|_{H^2}) \, ds\nn\\
			&\quad \le C\widetilde{E}_1(0) + C\widetilde{E}_1(t) \bigg\{ \int_{0}^{t} (1+s)^{-2} \, ds\bigg\}^{\frac{1}{2}} \bigg\{ \int_{0}^{t} (1+s)^{2} (\| \nabla^2 \u\|^2_{H^1} + \| \nabla^2 \tau\|^2_{H^2}) \, ds\bigg\}^{\frac{1}{2}} \nn\\
			& \quad\le C\widetilde{E}_1(0) + C\widetilde{E}_1(t) \widetilde{E}_2^{\frac{1}{2}} (t).
		\end{align}
Consequently, we complete the proof of the Lemma \ref{L2}.

\end{proof}

\subsection{The dissipation of the velocity field}
\begin{lemma}\label{L2}
		Let $(\u, \tau)$ be the solution of \eqref{reformulate1}, then there holds,  for any $t \ge 0$, that
			\begin{align}\label{energy4}
			\frac{1}{4}\int_{0}^{t} \| \nabla \u(s)\|^2_{H^2} \, ds \le \| (\u_0, \tau_0)\|^2_{H^3} + \| (\u(t), \tau(t))\|^2_{H^3}  + 3\int_{0}^{t} \| \nabla \tau\|^2_{H^3} \, ds + C\widetilde{E}_1(t) \widetilde{E}_2^{\frac{1}{2}} (t).
		\end{align}
\end{lemma}
\begin{proof}
		 Applying the  projection operator $\p$ to the first equation of \eqref{reformulate1} and the $\div $ operator to the second equation of \eqref{reformulate1}, we obtain the following new system:
		\begin{eqnarray}\label{reformulate2}
		\left\{\begin{aligned}
		&\partial_t \u   -  \p \div \tau =  -\p(\u \cdot \nabla \u),\\
		& \partial_t \div  \tau  -  \Delta  \div  \tau  - \frac{1}{2} \Delta \u = - \div (\u \cdot \nabla \tau) - \div ( g_{\alpha}(\tau, \nabla \u)).
		\end{aligned}\right.
		\end{eqnarray}
		Applying the  operator $\nabla^k( k=0,1,2)$ to the aforementioned system and then taking the $L^2$	inner products with $\nabla^k \tau$ and $\nabla^k \u$ respectively. Then, adding these two results together, we obtain
		\begin{align*}
			\frac{1}{2} \| \nabla^{k+1} \u\|^2_{L^2} &= - \frac{d}{dt} \langle \nabla^k \div  \tau, \nabla^k \u \rangle + \langle \nabla^k \div  \tau, \nabla^k \p \div  \tau \rangle - \langle \nabla^k \div ( \u \cdot \nabla\tau), \nabla^k \u\rangle\nn\\
			&\quad - \langle \nabla^k \div ( g_{\alpha}(\tau, \nabla\u)), \nabla^k \u\rangle + \langle \nabla^k \Delta \div  \tau, \nabla^k \u\rangle - \langle \nabla^k \div  \tau, \nabla^k \p(\u \cdot\nabla \u)\rangle.
		\end{align*}
		Next, summing the above equation over $0\le k\le 2$ gives rise to
		\begin{align}\label{energy2}
			\frac{1}{2} \| \nabla \u\|^2_{H^2} \le - \sum_{k=0}^{2}\frac{d}{dt} \langle \nabla^k \div  \tau, \nabla^k \u \rangle + \| \nabla \tau\|^2_{H^2} + \widetilde{K}_1 + \widetilde{K}_2 + \widetilde{K}_3 + \widetilde{K}_4.
		\end{align}
		It follows from the H\"older inequality that
		\begin{align}\label{K1}
			\widetilde{K}_1 &\le \|\u \|_{L^\infty}  \| \u\|_{H^1} \|\nabla \tau \|_{H^1}   + \| \nabla \u\|_{L^2} ( \| (\u, \nabla \u)\|_{L^\infty} \| \nabla^2 \tau\|_{H^1} + \|\nabla^2 \u \|_{L^2} \| \nabla \tau\|_{L^\infty})\nn\\
			&\lesssim \| (\u, \tau)\|^2_{H^3} (\| \nabla^2 \u\|_{H^1} + \| \nabla^2 \tau\|_{H^2}).
		\end{align}
		Similarly, by integration by parts, one can get
		\begin{align}\label{K2}
			\widetilde{K}_2 = \sum_{k=0}^{2} \langle \nabla^k ( g_{\alpha}(\tau, \nabla\u)), \nabla^{k+1} \u\rangle \lesssim \| (\u, \tau)\|^2_{H^3} (\| \nabla^2 \u\|_{H^1} + \| \nabla^2 \tau\|_{H^2}),
		\end{align}
		and
		\begin{align}\label{K4}
			\widetilde{K}_4 \lesssim \| (\u, \tau)\|^2_{H^3} (\| \nabla^2 \u\|_{H^1} + \| \nabla^2 \tau\|_{H^2}).
		\end{align}
		For the term $\widetilde{K}_3$, applying integration by parts and the Young inequality directly, we obtain
		\begin{align}\label{K3}
			\widetilde{K}_3 &= - \sum_{k=0}^{2} \langle \nabla^k \Delta  \tau, \nabla^{k+1} \u\rangle \le \| \nabla \u\|_{H^2} \| \nabla^2 \tau\|_{H^2}\nn\\
			& \le \frac{1}{4} \| \nabla \u\|^2_{H^2} + \| \nabla^2 \tau\|^2_{H^2}.
		\end{align}
		Substituting \eqref{K1}--\eqref{K3} into \eqref{energy2} leads to 
		\begin{align}\label{energy3}
			\frac{1}{4} \| \nabla \u\|^2_{H^2} \le - \sum_{k=0}^{2}\frac{d}{dt} \langle \nabla^k \div  \tau, \nabla^k \u \rangle + 3\| \nabla \tau\|^2_{H^2} + C\| (\u, \tau)\|^2_{H^3} (\| \nabla^2 \u\|_{H^1} + \| \nabla^2 \tau\|_{H^2}).
		\end{align}
		Integrating the above equation over the interval from $0$ to $t$ gives rise to \eqref{energy4}, which completes the proof.
	\end{proof}

\subsection{Time-weighted energy estimates}
\begin{lemma}\label{le_H3}
	Let $(\u, \tau)$ be the solution of \eqref{reformulate1}, then there holds that for any $t \ge 0$
	\begin{align}\label{E2}
		\widetilde{E}_2(t) \le C\widetilde{E}_2(0) + C\widetilde{E}_1(t) +C \widetilde{E}_2(t) \widetilde{E}_1^{\frac{1}{2}} (t).
	\end{align}
	\begin{proof}
		The proof of this lemma is divided into  two main steps. The first step is to estimate the basic energy in $\widetilde{E}_2(t)$ (which does not include the dissipation of $\u$). Firstly, applying operator $\nabla^k(k = 1,2,3)$ to \eqref{reformulate1}, and taking the $L^2$-inner product  with respect to $\nabla^k \u$ and $\nabla^k \tau$ in $\T^d$, respectively. Then adding them up gives rise to
		\begin{align}\label{H3_2}
		\frac{1}{2} \frac{d}{dt} \| (\nabla \u, \nabla \tau)\|^2_{H^2}  + \| \nabla^2 \tau\|^2_{H^2}  &=  \sum_{k=1}^{3} \langle \nabla^k \mathbf{g}_1, \nabla^k \u \rangle + \sum_{k=1}^{3} \langle \nabla^k \mathbf{g}_2, \nabla^k \tau \rangle  
		\end{align}
		where we have used  \eqref{fact1}.

		Multiplying \eqref{H3_2} by $(1+t)^2$   implies that
		\begin{align}\label{H3_3}
			& \frac{1}{2} \frac{d}{dt} (1+t)^2 \| (\nabla \u, \nabla \tau)\|^2_{H^2}  + (1+t)^2 \| \nabla^2 \tau\|^2_{H^2}\nn\\
			  &\quad= (1+t)\| (\nabla \u, \nabla \tau)\|^2_{H^2} - (1+t)^2 \sum_{k=1}^{3} \langle \nabla^k (\u \cdot \nabla \u), \nabla^k \u \rangle - (1+t)^2 \sum_{k=1}^{3} \langle \nabla^k (\u \cdot \nabla \tau), \nabla^k \tau \rangle\nn\\
			  & \quad\quad - (1+t)^2 \sum_{k=1}^{3} \langle \nabla^k (g_{\alpha}(\tau, \nabla \u)), \nabla^k \tau \rangle \nn\\
			  &\quad \stackrel{\mathrm{def}}{=} \widetilde{M}_1 + \widetilde{M}_2 + \widetilde{M}_3 + \widetilde{M}_4.
		\end{align}
		 For the first term $ \widetilde{M}_1$, due to the Poincar\'e inequality, we have
		 \begin{align}\label{L_1}
		 	&\int_{0}^{t} (1+s) \| (\nabla \u(s), \nabla \tau(s))\|^2_{H^2} \, ds\nn\\
&\quad\le 	C\int_{0}^{t} (1+s) ( \| \nabla \u(s)\|_{H^2} \| \nabla^3 \u(s)\|_{L^2} + \| \nabla \tau(s)\|_{H^2} \| \nabla^3 \tau(s)\|_{L^2})  \, ds \nn\\
		 	& \quad\le C\biggl\{ \int_{0}^{t} \| \nabla \u(s)\|^2_{H^2} \, ds\biggr\}^{\frac{1}{2}} \biggl\{ \int_{0}^{t}(1+s)^2 \| \nabla^2 \u(s)\|^2_{H^1} \, ds\biggr\}^{\frac{1}{2}} \nn\\
		 	& \quad\quad +C\biggl\{ \int_{0}^{t} \| \nabla \tau(s)\|^2_{H^3} \, ds\biggr\}^{\frac{1}{2}} \biggl\{ \int_{0}^{t}(1+s)^2 \| \nabla^2 \tau(s)\|^2_{H^2} \, ds\biggr\}^{\frac{1}{2}} \nn\\
		 	& \quad\le C \widetilde{E}_1^{\frac{1}{2}}(t) \widetilde{E}_2^{\frac{1}{2}}(t).
		 \end{align}
		 According to H\"older's inequality and embedding relations, one gets that
		 \begin{align}\label{L_2}
		 	\text{for} \quad d = 3, \quad \widetilde{M}_2 =& -(1+t)^2 \big( \langle \nabla \u \cdot \nabla \u, \nabla \u\rangle + 2  \langle \nabla \u \cdot \nabla^2 \u, \nabla^2 \u\rangle \nn\\
&+ 2 \langle \nabla \u \cdot \nabla^3 \u, \nabla^3 \u\rangle +  \langle \nabla^2 \u \cdot \nabla^2 \u, \nabla^3 \u\rangle\big) \nn\\
		 	\lesssim&  (1+t)^2 ( \|\u\|_{L^\infty} \|\nabla \u \|^2_{H^2} + \| \nabla^2 \u\|_{L^3} \| \nabla^2 \u\|_{L^6} \| \nabla^3 \u\|_{L^2})\nn\\
		 	\lesssim&  (1+t)^2 \| \u\|_{H^3} \| \nabla^2 \u\|^2_{H^1};\\
		 	\text{for} \quad d = 2, \quad \widetilde{L}_2 & \lesssim (1+t)^2 ( \|\u\|_{L^\infty} \|\nabla \u \|^2_{H^2} + \| \nabla^2 \u\|_{L^4} \| \nabla^2 \u\|_{L^4} \| \nabla^3 \u\|_{L^2})\nn\\
		 	\lesssim&  (1+t)^2 \| \u\|_{H^3} \| \nabla^2 \u\|^2_{H^1}.
		 \end{align}
		 Similarly, we can also obtain:
		 \begin{align}\label{L_3}
		 	\widetilde{M}_3 \lesssim (1+t)^2 \| (\u, \tau)\|_{H^3} \| \nabla^2 \tau\|^2_{H^2},
		 \end{align}
		 and
		 \begin{align}\label{L_4}
		 	\widetilde{M}_4 &= (1+t)^2 \sum_{k=1}^{3} \langle \nabla^{k-1} (g_{\alpha}(\tau, \nabla \u)), \nabla^{k+1} \tau \rangle\nn\\ &\lesssim (1+t)^2 \| (\u, \tau)\|_{H^3} (\| \nabla^2 \u\|^2_{H^1} +\| \nabla^2 \tau\|^2_{H^2}).
		 \end{align}
		 Integrating Equation \eqref{H3_3} over the interval from $0$ to $t$ and combining it with \eqref{L_1}--\eqref{L_4}, we obtain
		 \begin{align}\label{basic2}
		 	& \sup_{0 \le s \le t} (1+s)^2 \| (\nabla \u(s), \nabla \tau(s))\|^2_{H^2} + \int_{0}^{t} (1+s)^2  \|\nabla^2 \tau(s) \|^2_{H^2} \, ds\nn\\
		 	& \quad\le C\widetilde{E}_1(0) + C \widetilde{E}_1^{\frac{1}{2}}(t) \widetilde{E}_2^{\frac{1}{2}}(t) + C \widetilde{E}_1^{\frac{1}{2}}(t) \widetilde{E}_2(t).
		 \end{align}
		 \noindent { $\mathrm{\mathbf{Step 2.}}$} The second step is to estimate the supplementary dissipation in $\widetilde{E}_2(t)$ with respect to $\u$.
		 	Applying the $\nabla^k( k=1,2)$ operator to  system \eqref{reformulate2} and then taking the $L^2$	inner products with $\nabla^k \tau$ and $\nabla^k \u$ respectively. Then, adding these two results together, we obtain:
		 \begin{align*}
		 \frac{1}{2} \| \nabla^{k+1} \u\|^2_{L^2} &= - \frac{d}{dt} \langle \nabla^k \div  \tau, \nabla^k \u \rangle + \langle \nabla^k \div  \tau, \nabla^k \p \div  \tau \rangle - \langle \nabla^k \div ( \u \cdot \nabla\tau), \nabla^k \u\rangle\nn\\
		 &\quad - \langle \nabla^k \div ( g_{\alpha}(\tau, \nabla\u)), \nabla^k \u\rangle + \langle \nabla^k \Delta \div  \tau, \nabla^k \u\rangle - \langle \nabla^k \div  \tau, \nabla^k \p(\u \cdot\nabla \u)\rangle.
		 \end{align*}
		 Next, summing the above equation over $1\le k\le 2$ gives rise to
		 \begin{align}\label{energy}
		 \frac{1}{2} \| \nabla^2 \u\|^2_{H^1} \le& - \sum_{k=1}^{2}\frac{d}{dt} \langle \nabla^k \div  \tau, \nabla^k \u \rangle + \| \nabla^2 \tau\|^2_{H^1}+\sum_{k=1}^{2}\Big(- \langle \nabla^k \div ( \u \cdot \nabla\tau), \nabla^k \u\rangle\nn\\
&- \langle \nabla^k \div ( g_{\alpha}(\tau, \nabla\u)), \nabla^k \u\rangle + \langle \nabla^k \Delta \div  \tau, \nabla^k \u\rangle - \langle \nabla^k \div  \tau, \nabla^k \p(\u \cdot\nabla \u)\rangle\Big).
		 \end{align}
		 The last term in \eqref{energy} can be bounded  similarly to $\widetilde{K}_1$--$\widetilde{K}_4$,  one can get
		 \begin{align}\label{energy5}
		 	\frac{1}{4} \| \nabla^2 \u\|^2_{H^1} \le - \sum_{k=1}^{2}\frac{d}{dt} \langle \nabla^k \div  \tau, \nabla^k \u \rangle + 3\| \nabla^2 \tau\|^2_{H^1} + C\| (\u, \tau)\|_{H^3} (\| \nabla^2 \u\|^2_{H^1} + \| \nabla^2 \tau\|^2_{H^2}).
		 \end{align}
		 Multiplying the above expression by $(1+t)^2$ implies that
		 \begin{align}\label{energy6}
		 	& \frac{1}{4}  (1+t)^2 \| \nabla^2 \u\|^2_{H^1}  \nn\\
		 	&\quad\le - \sum_{k=1}^{2}\frac{d}{dt} (1+t)^2 \langle \nabla^k \div  \tau, \nabla^k \u \rangle +2(1+t) \sum_{k=1}^{2} \langle \nabla^k \div  \tau, \nabla^k \u \rangle   \nn\\
		 	&\quad\quad + 3(1+t)^2\| \nabla^2 \tau\|^2_{H^1} + C(1+t)^2\| (\u, \tau)\|_{H^3} (\| \nabla^2 \u\|^2_{H^1} + \| \nabla^2 \tau\|^2_{H^2}).
		 \end{align}
		 Integrating the above expression over the interval from $0$ to $t$, we have
		 \begin{align}\label{energy7}
		 	& \frac{1}{4} \int_{0}^{t} (1+s)^2 \| \nabla^2 \u(s)\|^2_{H^1}  \, ds \nn\\
		 	&\quad\le \widetilde{E}_2(0) + (1+t)^2 (\| \nabla \u\|^2_{H^1} + \| \nabla^2 \tau\|^2_{H^1}) + C \int_{0}^{t} (1+s) \| \nabla \u\|_{H^2} \| \nabla^2 \tau\|_{H^2} \, ds \nn\\
		 	&\quad\quad + 3 \int_{0}^{t} (1+s)^2 \| \nabla^2 \tau\|^2_{H^2}\, ds + C \widetilde{E}_1^{\frac{1}{2}}(t) \widetilde{E}_2(t)\\
		 	& \quad\le \widetilde{E}_2(0) + C\widetilde{E}_1^{\frac{1}{2}} (t) (\widetilde{E}_2^{\frac{1}{2}} (t) + \widetilde{E}_2(t)) + (1+t)^2 \| (\nabla \u, \nabla^2 \tau)\|^2_{H^1} + 3\int_{0}^{t} (1+s)^2 \| \nabla^2 \tau\|^2_{H^2}\, ds \nn
		 \end{align}
		 where we have used the fact
		 \begin{align*}
		 	\int_{0}^{t} (1+s) \| \nabla \u\|_{H^2} \| \nabla^2 \tau\|_{H^2} \, ds &\le \biggl\{\int_{0}^{t} \| \nabla \u\|^2_{H^2}  \, ds \biggr\}^{\frac{1}{2}} \biggl\{\int_{0}^{t} (1+s)^2 \| \nabla^2 \tau\|^2_{H^2}  \, ds \biggr\}^{\frac{1}{2}}\nn\\
		 	& \le \widetilde{E}_1^{\frac{1}{2}} (t) \widetilde{E}_2^{\frac{1}{2}} (t).
		 \end{align*}
		 Multiplying \eqref{basic2} by $4$ and adding it to \eqref{energy7}  gives
		 \begin{align}\label{energy8}
		 	\widetilde{E}_2(t) &\le C\widetilde{E}_2(0) + C\widetilde{E}_1^{\frac{1}{2}} (t) (\widetilde{E}_2^{\frac{1}{2}} (t) + \widetilde{E}_2(t))\nn\\
		 	& \le C\widetilde{E}_2(0) + \frac{1}{2}\widetilde{E}_2(t) + C\widetilde{E}_1(t) + C\widetilde{E}_1^{\frac{1}{2}} (t) \widetilde{E}_2(t)
		 \end{align}
		 which implies \eqref{E2}, this completes the proof of Lemma \ref{le_H3}.
		\end{proof}
\end{lemma}

\subsection{Completion the proof of Theorem \ref{theorem1.2}}
With the a priori estimates in the previous section, we shall use the continuity argument to complete the proof of the global well posedness of Theorem \ref{theorem1.2}. At first, we define $ \widetilde{E}_{total}(t)$ by
\begin{align}\label{Etotal}
 \widetilde{E}_{total}(t) \stackrel{\mathrm{def}}{=} \widetilde{E}_1(t) + \widetilde{E}_2(t).
\end{align}
		Multiplying \eqref{tauguji1} by $4$ and adding it to \eqref{energy4}  implies 
	\begin{align}\label{E1}
			\widetilde{E}_1(t) \le C\widetilde{E}_1(0) + C \widetilde{E}_1(t) \widetilde{E}_2^{\frac{1}{2}} (t).
	\end{align}

Multiplying Equation \eqref{E1} by $2C$ and adding it to Equation \eqref{E2}, we obtain:
\begin{align}\label{Etotal1}
	\widetilde{E}_{total}(t) \le C\widetilde{E}_{total}(0) + C \widetilde{E}_{total}^{\frac{3}{2}}(t).
\end{align}

According to the theory of local well-posedness of \eqref{reformulate1}, applying the bootstrapping argument to the above equation, this completes the proof of the global well--posedness part of Theorem \ref{theorem1.2}. Here we omit the details for convenience.

\section{The proof of the decay rates of Theorem \ref{theorem1.2}}
In this section, we establish the exponential decay of solutions constructed in the preceding analysis. While the weighted energy method successfully yields global well-posedness for System \eqref{reformulate1} and provides decay estimates for higher-order derivatives, the resulting algebraic decay rates are suboptimal. To overcome this limitation, we strategically introduce the effective viscous flux
\begin{align}\label{youxiao}
\mathbf{G}\stackrel{\mathrm{def}}{=} -\q \tau - \frac{1}{2} \nabla \Delta^{-1} \u
\end{align}
a carefully designed quantity that crucially enables enhancement of the decay regime. This construct allows us to refine the energy framework, ultimately achieving exponential decay for both the velocity field $\u$ and the stress tensor $\tau$.

We first get by a similar derivation
 of \eqref{basic2} that
\begin{align}\label{half2}
	\frac{1}{2} \frac{d}{dt} \| (\nabla \u, \nabla \tau)\|^2_{H^2} +   \|\nabla^2 \tau \|^2_{H^2} &\le C \| (\u, \tau)\|_{H^3} ( \| \nabla \u\|^2_{H^2} + \| \nabla \tau\|^2_{H^3})\nn\\
	& \le C \| (\u, \tau)\|_{H^3}  \| \nabla \u\|^2_{H^2}
\end{align}
in which we use the following fact:$$ \| \nabla \tau\|^2_{H^3} \le C \| \nabla^2 \tau\|^2_{H^2}  \quad \text{and} \quad \| (\u, \tau)\|_{H^3} \le C\varepsilon.$$

Thanks to the definition of $\mathbf{G}$, it is direct to see that $\u$ satisfy the following damping equation
\begin{align}\label{u}
&\partial_t \u + \frac{1}{2} \u = -\mathbb{P} (\u \cdot \nabla \u) - \mathbb{P} \div \mathbf{G}.
\end{align}
For the aforementioned equation, using standard energy estimates, we obtain, for $ 1 \le k \le 3$, that
\begin{align*}
\frac{1}{2} \frac{d}{dt} \| \Lambda^k \u\|^2_{L^2} + \frac{1}{2} \| \Lambda^k \u\|^2_{L^2}  &= -\int_{\T^d} \Lambda^k \mathbb{P} (\u \cdot \nabla \u) \cdot \Lambda^k \u\, dx  -\int_{\T^d} \Lambda^k \mathbb{P} \div \mathbf{G} \cdot \Lambda^k \u\, dx.
\end{align*}
The nonlinear estimates in the above expression are similar to those in $\widetilde{J}_1$  and $\widetilde{J}_2$, hence, we can get
\begin{align*}
\frac{1}{2} \frac{d}{dt} \| \Lambda^k \u\|^2_{L^2} + \frac{1}{4} \| \Lambda^k \u\|^2_{L^2} \le C\| \Lambda^{k+1} \mathbf{G}\|^2_{L^2} + C\| \u\|_{H^3} \| \Lambda^k \u\|^2_{L^2}
\end{align*}
from which and
summing up $1\le k \le 3$ give rise to
\begin{align}\label{damp3}
	\frac{1}{2} \frac{d}{dt} \| \nabla \u\|^2_{H^2} + \frac{1}{4} \| \nabla \u\|^2_{H^2} \le C \| \nabla^2 \mathbf{G}\|^2_{H^2} + C \|\u\|_{H^3} \|\nabla \u \|^2_{H^2}.
\end{align}
It can be directly verified that $\mathbf{G}$ satisfies the following equation:
\begin{align}\label{omega}
	\partial_t \mathbf{G} - \Delta \mathbf{G} =  -\nabla \Delta^{-1} \big(\frac{1}{2} \p \div(\u \otimes \u)  + \div \mathbf{g}_2\big) -\frac{1}{2} \nabla \Delta^{-1} \p\div \tau.
\end{align}
Applying the operator $ \Lambda^k(1 \le k \le 3)$ to the above equation and then taking the $L^2$ inner product with $ \Lambda^k \mathbf{G}$, we obtain
\begin{align}\label{damp4}
	\frac{1}{2} \frac{d}{dt} \| \Lambda^k \mathbf{G}\|^2_{L^2} + \| \nabla \Lambda^k \mathbf{G}\|^2_{L^2} &\le \Big| \int_{\T^d}  \Lambda^k \nabla \Delta^{-1} \p \div(\u \otimes \u) \cdot \Lambda^k \mathbf{G}  \, dx\Big| + \Big| \int_{\T^d}  \Lambda^k \q \mathbf{g}_2 \cdot \Lambda^k \mathbf{G}  \, dx\Big| \nn\\
	&\quad + \Big| \int_{\T^d}  \Lambda^k \nabla \Delta^{-1} \p \div \tau  \cdot \Lambda^k \mathbf{G}  \, dx\Big| \\
	& \le (\|  \Lambda^k (\u \otimes \u)\|_{L^2} + \| \Lambda^k \tau\|_{L^2}) \| \Lambda^k \mathbf{G}\|_{L^2} + \|  \Lambda^{k-1} \mathbf{g}_2\|_{L^2} \| \nabla \Lambda^k \mathbf{G}\|_{L^2}.\nn
\end{align}
Summing the above equation over $1\le k \le 3$, and then applying Young's inequality give rise to
\begin{align}\label{damp5}
	\frac{1}{2} \frac{d}{dt} \| \nabla \mathbf{G}\|^2_{H^2} + \frac{1}{4} \| \nabla^2 \mathbf{G}\|^2_{H^2} \le C (\| \nabla (\u \otimes \u)\|^2_{H^2} + \|\nabla \tau\|^2_{H^2} + \|\mathbf{g}_2\|^2_{H^2} )
\end{align}
where we have used the fact $$ \|\nabla \tau\|_{H^2} \|\nabla \mathbf{G} \|_{H^2} \le \frac{1}{4} \|\nabla^2 \mathbf{G} \|^2_{H^2}  + C \|\nabla \tau\|^2_{H^2}.$$
The left work is  to estimate  nonlinear terms in \eqref{damp5}. Since $H^2(\T^d)$ is a Banach algebra, we obtain
\begin{align}\label{nonlinear}
	\| \nabla (\u \otimes \u)\|^2_{H^2} \lesssim \| \u\|^2_{H^3} \| \nabla \u\|^2_{H^2}, \quad
	 \|\mathbf{g}_2\|^2_{H^2} \lesssim \| (\u, \tau)\|^2_{H^3} \| (\nabla \u, \nabla \tau)\|^2_{H^2}.
\end{align}

The combination of \eqref{damp5} and \eqref{nonlinear} leads to
\begin{align}\label{damp6}
	\frac{1}{2} \frac{d}{dt} \| \nabla \mathbf{G}\|^2_{H^2} + \frac{1}{4} \| \nabla^2 \mathbf{G}\|^2_{H^2} &\le C \|\nabla \tau\|^2_{H^2} + C \| (\u, \tau)\|^2_{H^3} \| (\nabla \u, \nabla \tau)\|^2_{H^2}\nn\\
	& \le C \|\nabla^2 \tau\|^2_{H^2} + C \| (\u, \tau)\|^2_{H^3} \| (\nabla \u, \nabla \tau)\|^2_{H^2}.
\end{align}

Multiplying \eqref{damp6} by $8C$ and adding it to \eqref{damp3}, we have
\begin{align}\label{damp7}
	 \frac{d}{dt} \| (\nabla \u, \nabla \mathbf{G})\|^2_{H^2} + c(\| \nabla \u\|^2_{H^2} + \| \nabla^2 \mathbf{G}\|^2_{H^2}) \le C \|\nabla^2 \tau\|^2_{H^2} + C \| (\u, \tau)\|^2_{H^3} \| (\nabla \u, \nabla \tau)\|^2_{H^2}.
\end{align}

Multiplying \eqref{half2} by $2C$ and adding it to \eqref{damp7}, we have
\begin{align*}
	\frac{d}{dt} \| (\nabla \u, \nabla \tau, \nabla \mathbf{G})\|^2_{H^2} + c(\| (\nabla \u, \nabla^2 \tau, \nabla^2 \mathbf{G})\|^2_{H^2} \le C \| (\u, \tau)\|^2_{H^3} \| (\nabla \u, \nabla \tau)\|^2_{H^2}.
\end{align*}
According to the global well--posedness part of Theorem \ref{theorem1.2} and the Poincar\'e inequality, one can further get
\begin{align*}
	\frac{d}{dt} \| (\nabla \u, \nabla \tau, \nabla \mathbf{G})\|^2_{H^2} + c(\| (\nabla \u, \nabla \tau, \nabla \mathbf{G})\|^2_{H^2} \le 0.
\end{align*}
Solving the above differential equation, we obtain:
\begin{align}\label{high_decay}
	\| (\nabla \u, \nabla \tau, \nabla \mathbf{G})\|_{H^2} \le  \widetilde{C_0} e^{-\widetilde{C_1}t}.
\end{align}
Next, we need to obtain the exponential decay of the $L^2$-norm of $\u$, taking the $L^2$ inner product of \eqref{u} with $\u$ leads to
\begin{align}\label{u1}
	\frac{1}{2} \frac{d}{dt} \| \u\|^2_{L^2} + \frac{1}{2} \| \u\|^2_{L^2} = - \langle \p \div \mathbf{G}, \u \rangle
\end{align}
where we have used the fact $$ \langle \p(\u \cdot \nabla \u),\u \rangle = 0.$$
Then applying Young's inequality $ |\langle \p \div \mathbf{G}, \u \rangle| \le \frac{1}{4} \| \u\|^2_{L^2} + \| \nabla \mathbf{G} \|^2_{L^2}$ gives rise to
\begin{align}\label{u2}
	\frac{d}{dt} \| \u\|^2_{L^2} + \frac{1}{2} \| \u\|^2_{L^2} \le \| \nabla \mathbf{G} \|^2_{L^2}.
\end{align}
Applying Gronwall's inequality to \eqref{u2}, we have:
\begin{align*}
	\| \u\|^2_{L^2} & \le \| \u_0\|^2_{L^2} e^{-\frac{1}{2}t} + \int_{0}^{t} e^{-\frac{1}{2}(t-s)} \| \nabla \mathbf{G} \|^2_{L^2} \, ds\nn\\
	& \le \| \u_0\|^2_{L^2} e^{-2C_3 t} + \widetilde{C_0}^2 e^{ -2C_3 t}\int_{0}^{t} e^{2C_3 s} e^{-2\widetilde{C_1} s} \, ds \nn\\
	& \le C_2^2 e^{-2C_3 t},
\end{align*}
where $C_3 = \min \Big\{ \frac{\widetilde{C_1}}{2}, \frac{1}{4} \Big\}$ and $ C_2 = \min \Big\{ \| \u_0\|_{L^2}, \Big( \frac{\widetilde{C_0}^2}{2C_3}\Big)^{\frac{1}{2}}, \widetilde{C_0}  \Big\}.$
This completes the proof of Theorem \ref{theorem1.2}. \hspace{14.6cm}
$\square$

\bigskip

 \section*{Acknowledgement}
 {This work is  supported by the Guangdong Provincial Natural Science Foundation under grant  2024A1515030115.}

 \vskip .1in
\noindent{\bf Data Availability Statement} Data sharing is not applicable to this article as no
data sets were generated or analysed during the current study.

\vskip .1in

\noindent{\bf Conflict of Interest} The authors declare that they have no conflict of interest. The
authors also declare that this manuscript has not been previously published, and
will not be submitted elsewhere before your decision.

\end{document}